\documentclass[12pt, reqno]{amsart}

\usepackage[top=25mm, bottom=20mm, left=24mm, right=24mm]{geometry}

\usepackage{amsmath, amsthm, amssymb}
\usepackage{hyperref}
\usepackage{enumerate}
\usepackage{url}

\usepackage[T1]{fontenc}
\usepackage{lmodern}
\usepackage{mathtools,mathrsfs,microtype}

\usepackage{xcolor}

\usepackage{fancyvrb}

\newtheorem{thm}{Theorem}[section]
\newtheorem{lemma}[thm]{Lemma}
\newtheorem{prop}[thm]{Proposition}

\newtheorem{conj}[thm]{Conjecture}

\theoremstyle{definition}
\newtheorem{defn}[thm]{Definition}

\theoremstyle{remark}
\newtheorem{remark}[thm]{Remark}

\numberwithin{equation}{section}

\def\alp{{\alpha}}

\def\del{{\delta}} \def\Del{{\Delta}}
 
\def\tet{{\theta}}  
\def\kap{{\kappa}}

\def\ome{{\omega}}

\def \le {\leqslant} \def\ge{\geqslant}

\def \ind {{\mathrm{ind}}}

\def \bE {\mathbb E}
\def \bF {\mathbb F}

\def \bN {\mathbb N}

\def \bR {\mathbb R}
\def \bZ {\mathbb Z}

\def \bb {\mathbf b}
\def \bc {\mathbf c}

\def \bx {\mathbf x}

\def \bbet {{\boldsymbol{\beta}}}

\def \fG {\mathfrak G}

\def \cA {\mathcal A}

\def \cH {\mathcal H}

\def \cM {\mathcal M}

\def \cR {\mathcal R}

\def \cU {\mathcal U}
\def \cV {\mathcal V}

\def \dim {\mathrm{dim}}
\def \ord {\mathrm{ord}}

\def \deg {\operatorname{deg}}

\def \TC {{\mathrm{TC}}}

\newcommand{\Z}{\mathbf Z}
\newcommand{\Q}{\mathbf Q}
\newcommand{\C}{\mathbf C}
\newcommand{\Fp}{\mathbf F_p}
\newcommand{\Ftwo}{\mathbf F_2}
\newcommand{\A}{\mathbf A}
\newcommand{\PP}{\mathbf P}
\newcommand{\OO}{\mathcal O}
\newcommand{\LL}{\mathcal L}
\newcommand{\DD}{\mathscr D}
\DeclareMathOperator{\Disc}{\Delta}
\DeclareMathOperator{\Map}{Map}

\DeclareMathOperator{\Int}{int}
\DeclareMathOperator{\GL}{GL}
\DeclareMathOperator{\AGL}{AGL}

\renewcommand{\geq}{\ge}
\renewcommand{\leq}{\le}
\renewcommand{\epsilon}{\varepsilon}

\begin{document}
\title[{\tiny Strong Van der Waerden via Chowla}]{The strong form of Van der Waerden's conjecture via twisted Chowla}



\author[{\tiny Anderson}]{Theresa C. Anderson}
\address{Department of Mathematical Sciences, Carnegie Mellon University: Wean Hall, Hammerschlag Dr., Pittsburgh, PA 15213, USA}
\email{tanders2@andrew.cmu.edu}

\author[{\tiny Chow}]{Sam Chow}
\address{Mathematics Institute, Zeeman Building, University of Warwick, Coventry CV4 7AL, United Kingdom}
\email{Sam.Chow@warwick.ac.uk}

\author[{\tiny Dietmann}]{Rainer Dietmann}
\address{Department of Mathematics, Royal Holloway, University of London\\
Egham TW20 0EX, United Kingdom}
\email{Rainer.Dietmann@rhul.ac.uk}

\author[{\tiny Hokken}]{David Hokken}
\address{Max-Planck-Institut f\"ur Mathematik, Vivatsgasse 7, D-53111 Bonn, Germany}
\email{hokken@mpim-bonn.mpg.de}

\author[{\tiny Koymans}]{Peter Koymans}
\address{Mathematisch Instituut, Postbus 80.010, 3508 TA Utrecht, Netherlands}
\email{p.h.koymans@uu.nl}

\author[{\tiny Lemke Oliver}]{Robert J. Lemke Oliver}
\address{Department of Mathematics, University of Wisconsin-Madison, Madison, WI 53706, USA}
\email{lemkeoliver@wisc.edu}

\author[{\tiny Sawin}]{Will Sawin}
\address{Department of Mathematics, Princeton University, Fine Hall, Washington Road, Princeton, NJ 08540 USA}
\email{wsawin@math.princeton.edu}

\author[{\tiny Shusterman}]{Mark Shusterman}
\address{The Dr. A. Edward Friedmann Career Development Chair in Mathematics at the Faculty of Mathematics and Computer Science, Weizmann Institute of Science, 234 Herzl Street, Rehovot 76100, Israel}
\email{mark.shusterman@weizmann.ac.il}

\begin{abstract} 
Determining the properties of a random polynomial has fuelled significant investigation over the past century.  One driving force of this research is a 1936 paper of Van der Waerden. Fix $n \ge 3$ and let $E_n(B)$ be the number of monic, irreducible, non-$S_n$ polynomials $f = X^n + a_1 X^{n-1} + \cdots + a_n$ with \mbox{$|a_j| \leq B$} for all $j$.
A recent breakthrough of Bhargava bounds \mbox{$E_n(B) \ll B^{n-1}$.}  This spectacularly resolves a conjecture of Van der Waerden, but leaves open its stronger form, namely that $E_n(B) = o(B^{n-1})$.
Inspired by recent progress, we now prove this stronger form for all $n \ge 40$.

Bhargava's result, together with work of Chow and Dietmann, essentially reduces the strong Van der Waerden conjecture to the claim that the number of polynomials $f$ with Galois group $A_n$ is $o(B^{n-1})$. We connect this claim to Chowla's conjecture, an area driving major mathematical progress. 
Indeed, we formulate a twisted function field version of Chowla’s conjecture, and show how Van der Waerden's conjecture reduces to this conjecture for $n \ge 8$. Finally, we make significant progress towards this twisted Chowla conjecture, leading to our resolution of the strong form of Van der Waerden's conjecture for all $n \geq 40$.
\end{abstract}

\maketitle

\section{Introduction}

\subsection{Galois groups of random polynomials}

Fix an integer $n \ge 3$. Given $f(X) \in \bZ[X]$ of degree $n$, its Galois group $G_f$ is the automorphism group of the splitting field of $f$. If $f$ is separable, then $G_f$ acts on the $n$ many distinct roots of $f$. Thus, up to conjugacy, we may regard $G_f$ as a subgroup of $S_n$. Let $B \geq 1$. It follows from the Hilbert irreducibility theorem that among polynomials 
\begin{equation}
\label{monic}
f(X) = X^n + a_1X^{n-1} + \cdots + a_n \in \bZ[X]
\end{equation}
with all $|a_j| \le B$, almost all are irreducible with $G_f \cong S_n$ as $B \to \infty$.  

In a 1936 paper, Van der Waerden \cite{vdW1936} was the first to consider quantitative forms of this theorem, referring to polynomials $f$ such that $G_f \not\cong S_n$ as `polynomials with affect'.  He obtained bounds both on the number of reducible polynomials and on the number of irreducible polynomials with affect. Van der Waerden then notes that a polynomial of degree $2$ has affect if and only if it is reducible, and that the count of such is $O(B\log{B})$. He then writes:

\begin{quote}
    \em ``Ich vermute, da{\ss} die H\"aufigkeit der Polynome mit Affekt auch f\"ur $n>2$ im wesentlichen gleich der H\"aufigkeit der reduziblen Polynome ist.  Es scheint n\"amlich, da{\ss} die irreduziblen Polynome mit Affekt noch erheblich seltener sind als die reduziblen Polynome.  Diese Vermutung wird best\"atigt durch die heuristische Betrachtung des folgenden Paragraphen, welche es plausibel macht, da{\ss} die H\"aufigkeit der Gleichungen 3.~Grades mit alternierender Gruppe $O(N^{-2+\epsilon})$ ist.''
\end{quote}
Translated to English, this reads:
\begin{quote}
    \em ``I suspect that the frequency of polynomials with affect is, also for $n > 2$, essentially the same as the frequency of reducible polynomials.  Indeed, it seems that irreducible polynomials with affect are significantly rarer than reducible polynomials. This conjecture is supported by the heuristic consideration of the following paragraph, which makes it plausible that the frequency of cubic equations with alternating group is $O(N^{-2+\epsilon})$.''
\end{quote}
From this, it is apparent that Van der Waerden believed two things:
\begin{enumerate}
    \item The count of polynomials with $G_f \not\cong S_n$ should be dominated by the number of reducible polynomials,\footnote{Strictly speaking, Van der Waerden proved the analogous result for non-monic polynomials and used the variable $N$ in place of our $B$.  The first discrepancy is almost as immaterial as the second!} which he showed to be $\asymp B^{n-1}$.
    \item The number of irreducible polynomials with $G_f \not\cong S_n$ should be little-o of the number of reducible polynomials, i.e. $o(B^{n-1})$, and likely with a power saving.
\end{enumerate}

In work published in 2025, Bhargava \cite{Bha2025} spectacularly showed that the first of these assertions is correct: the error term in the form of the Hilbert irreducibility theorem above is $O(B^{n-1})$.  Bhargava's methods also naturally bear on the second assertion, even if they do not resolve it. This second assertion, which we refer to as the strong form of Van der Waerden's conjecture, is the main focus of this paper, so we state it precisely.

\begin{conj}[The strong form of Van der Waerden's conjecture]
\label{vdW}
Let $n \ge 3$ be an integer. For any $B \geq 1$, let $E_n(B)$ be the number of irreducible polynomials
    \[
        f(X)
            = X^n + a_1X^{n-1} + \dots + a_n \in \bZ[X]
    \]
such that $|a_j| \leq B$ for each $1 \leq j \leq n$ and such that $G_f \not\cong S_n$.  Then $$E_n(B) = o_n(B^{n-1})$$
as $B \to \infty$.
\end{conj}

Write $A_n$ for the alternating group on $n$ letters. In many problems in probabilistic Galois theory, finding good bounds for $A_n$ is a major obstacle to finding bounds for non-$S_n$ groups. This is indeed the case for us; as detailed below, the remaining challenge to obtain Conjecture~\ref{vdW} is to show that $G_f \cong A_n$ with probability $o_n(B^{-1})$. Our present work resolves that problem for large degrees, and our main theorem is as follows:

\begin{thm}
\label{intro-main} 
We have $E_n(B) = O_n( B^{n-1-\frac{1}{440}})$ for all $n \ge 40$. In particular, Conjecture \ref{vdW} holds for $n\geq 40$. 
\end{thm}

Moreover, the power saving is greater for large $n$, as we will detail below.

The proof of Theorem~\ref{intro-main} hinges on two further theorems. First, conditionally on a Chowla-type exponential sum estimate (described in more detail below), Theorem~\ref{MainThm} establishes Conjecture~\ref{vdW} with a power saving (where we note that the currently specified saving of $1/440$ may be replaced by $1/14 - \varepsilon$ for sufficiently large values of $n$). Theorem~\ref{ChowlaMain} then establishes sufficient cancellation in the Chowla-type sums for $n \ge 40$.
It may be possible to obtain cancellation for smaller values of $n$ and thus use Theorem~\ref{MainThm} to obtain Conjecture~\ref{vdW} for an even wider range of $n$. As detailed below, if square root cancellation is achieved in the relevant exponential sums, this would resolve Conjecture~\ref{vdW} for all $n \geq 8$.  This reflects a common theme that lower degree estimates are difficult and require different techniques, such as detailed in \cite{CD2020}, \cite{Bha2025} and (in the updated version of) \cite{ABO2026}.  We will make the connection to Chowla's conjecture precise before discussing relevant previous works and the reduction to the case $G_f \cong A_n$ in more detail.

\subsection{Chowla's conjecture and the main result}
As indicated above, the main purpose of this paper is to prove the strong Van der Waerden conjecture for large $n$. We achieve this by connecting Van der Waerden's conjecture to a certain twisted version of Chowla's conjecture over function fields.  This connection brings together two deep and active areas of mathematics.  We begin by briefly recalling the standard Chowla conjecture, which comes in several forms, see \cite{Ram2018}. Denoting by $\mu$ the M\"obius function, a representative special case asserts that if $0 \ne h \in \bZ$ then
\[
M^{-1} \sum_{m \le M} \mu(m) \mu(m+h) \to 0 \qquad (M \to \infty).
\]
Chowla's conjecture is a major open problem in its own right, and is closely related to a far-reaching conjecture of Sarnak, see \cite{GKL2018}.

The function field analogue of Chowla's conjecture was introduced by Carmon and Rudnick~\cite{CR2014}. Suppose $q$ is a prime power, and that a nonzero polynomial $f(x) \in \bF_q[x]$ factors as $f = c f_1 \cdots f_r$, where the $f_j$ are monic and irreducible and $c$ is a unit. Then the M\"obius function of $f$ is defined as
\begin{equation}
\label{eq:mobius}
\mu_q(f) := \begin{cases}
    (-1)^r &\text{if the } f_j \text{ are pairwise coprime}, \\
    0 & \text{else.}
\end{cases}
\end{equation}
For any ring $R$, denote by $\cM_n(R)$ the set of monic polynomials in $R[X]$ of degree $n$, and write $\cA_{\le m}(R)$ for the set of all --- not necessarily monic --- polynomials in $R[X]$ of degree at most $m$. Let $\bE$ denote expectation. A special case of Carmon and Rudnick's result is that if $n \ge 2$ and $q$ is odd then
\[
\displaystyle
\sup_{0 \ne h \in \cA_{\le n - 1}(\bF_q)} \left| \bE_{g \in \cM_n(\bF_q)} \, \mu_q(g) \mu_q(g+h) \right|
\ll_n q^{-1/2}.
\]
Keating and Roditty-Gershon \cite{KR2015, KR2018} worked on this and similar problems, making quantitative progress. We refer the reader to the article of Gorodetsky and Sawin \cite{GS2020} who showed, \emph{inter alia}, that if $n \ge 4$ then
\[
\displaystyle
\sup_{0 \ne h \in \bF_q} \left| \bE_{g \in \cM_n(\bF_q)} \, \mu_q(g) \mu_q(g+h) \right|
\ll_n q^{-1}.
\]

We will be interested in similar statements that restrict to prime values $q = p$, but are more general in the sense that they involve an arbitrary twist by an exponential phase. Here and throughout, we write $e(x) = e^{2 \pi i x}$ for $x \in \bR$, and abbreviate $e_m(x) = e(x/m)$. 
Let $g \cdot \theta$ denote the dot product of monic polynomials $g(X) = X^n + b_1 X^{n-1} + \cdots + b_n$ and $\theta(X) = X^n + c_1 X^{n-1} + \cdots + c_n$, defined as $g \cdot \theta = \bb \cdot \bc$. We now define our main bridge between the Chowla and Van der Waerden conjectures: the statement $\TC(n, \tau)$ below.

\begin{defn}
\label{def:TC}
Fix $n \in \bN$ and $\tau > 0$. Then $\TC(n, \tau)$ is the statement that, for all sufficiently large primes $p$,
\begin{equation}
\label{eq:TC}
\displaystyle
\sup_{\substack{0 \ne h \in \cA_{\le n - 1}(\bF_p) \\ \tet \in \cM_n(\bF_p)}} \left|
\bE_{g \in \cM_n(\bF_p)} \, \mu_p(g) \mu_p(g+h) e_p(g \cdot \tet)
\right|
\ll_{n,\tau} p^{-\tau}.
\end{equation}
\end{defn}

The two ingredients to our main result are as follows. 

\begin{thm}
\label{MainThm}
Fix $n \ge 7$. Suppose that the statement $\TC(n, \tau)$ holds for some $\tau$ in the range
\begin{equation}
\label{tauRange}
    \frac{4n}{n+2} < \tau \leq \frac{n}{2}.
\end{equation}
Then $E_n(B) = O(B^{n-1-\delta(n,\tau)})$ for some $\delta(n,\tau)>0$ depending at most on $n$ and $\tau$, so that Conjecture~\ref{vdW} holds in degree $n$. More precisely, we 
have
\begin{equation}
\label{eq:EnBbound}
E_n(B) \ll_{n, \tau} B^{n-1- \frac{n (\tau - 4) + 2 \tau}{2 (n (n+2) - \tau(n+1))}} (\log{B})^2
\end{equation}
as $B \to \infty$. 
\end{thm}

\begin{thm}
\label{ChowlaMain}
Let $k$ be a positive integer. If $n\ge 4k+8$, then $\TC(n,\frac{k}{2} )$ holds.
\end{thm}

The proofs of Theorems~\ref{MainThm} and \ref{ChowlaMain} are given in \S \ref{passage} and \S \ref{sec:Chowla-estimates}, respectively. The deduction of Theorem~\ref{intro-main} from the two theorems is also provided in \S \ref{passage}.

Theorem~\ref{ChowlaMain} yields $\TC(n,4)$ for $n\ge40$, so that Conjecture~\ref{vdW} holds in those degrees. 
The more precise claim made in Theorem~\ref{intro-main} follows from a short minimisation argument, where we take the integer $k$ as large as possible while preserving the inequality $n \ge 4k+8$.

On account of the Plancherel theorem, the statement $\TC(n, \tau)$ does not hold for any $\tau > n/2$. We guess, based on the `square root cancellation' heuristic for character sums (using Parseval's inequality to obtain a bound on average), that $\TC(n, \tau)$ holds for $\tau = (n-1)/2$. (In the special case $\tet = 0$, the sum is essentially $(n-1)$-dimensional, after accounting for translation invariance.)

\begin{conj}
    For all $n \geq 3$, $\TC(n, \frac{n-1}{2})$ holds.
\end{conj}

If this is the case, then the hypothesis of Theorem \ref{MainThm} would be true for all $n \ge 8$.  

Conjecture \ref{vdW}, despite its apparent depth, is expected to be far from the truth. We refer the reader to \cite{BBM} for some interesting conjectures and questions that are likely to be closer to the boundary between what is true and what is not.

\subsection{History and reduction to the alternating group}
As already mentioned, Van der Waerden did not restrict to monic polynomials. However, the mantle was later taken up by Knobloch \cite{Kno1955} and Gallagher \cite{Gal1973} with methods applying to both versions of the problem (monic and non-monic). Gallagher's work relied on his development of the multidimensional large sieve, and gave the bound
$
E_n(B) \ll B^{n-1/2} \log B.
$
This was the benchmark for some 40 years before being refined by Zywina \cite{Zyw2010}, Dietmann~\cite{Die2013}, and more recently by Anderson--Gafni--Lemke Oliver--Lowry-Duda--Shakan--Zhang~\cite{Six2023} and Bhargava \cite{Bha2025}. In particular, Bhargava \cite{Bha2025} spectacularly showed that
\[
E_n(B) \ll B^{n-1}.
\]
Bhargava's exciting result, referred to as Van der Waerden's conjecture, utilised several new techniques, including the novel introduction of the \emph{double discriminant}.  Additionally, this result achieved the \emph{weaker form of Van der Waerden's conjecture}, the bound $E_n(B) \ll B^{n-1+\varepsilon}$, using a short, slick, and inventive argument, and further went on to `remove the epsilon', achieving \emph{Van der Waerden's conjecture}, via Fourier analytic tools.  However the \emph{strong form of Van der Waerden's conjecture} remained open for all $n \neq 3,4$; that is, Conjecture~\ref{vdW} has only been resolved for $n=3,4$ (by Chow and Dietmann \cite{CD2020}).

We know from \cite{Bha2025, CD2023} that stronger bounds than Bhargava's $E_n(B) \ll B^{n-1}$ can be obtained if one also excludes the alternating group. For example, if we let $E_n^\dag(B)$ count irreducible polynomials \eqref{monic} such that $|a_j| \le B$ for all $j$ and $G_f \not \ge A_n$, then combining \cite[Corollary 3(a)]{Bha2025} and \cite[Equation (5)]{Bha2025} --- the latter due to Widmer \cite{Wid2011} via the correspondence \cite[Lemma 2.1]{CD} --- yields
\begin{equation}
\label{eq:EnBn-2}
E_n^\dag(B) \ll B^{n-2}
\qquad (n \ge 10).
\end{equation}
Additionally, for large $n$, \cite[Corollary~1.3]{LO} gives the bound $E_n^\dag(B) \ll B^{n - c n/\log n}$ for a constant $c>0$.

Combining \eqref{eq:EnBn-2} with \cite[Theorem 1.7]{CD2023} yields
\begin{equation}\label{AltMain}
E_n^\dag(B) \ll B^{n-1.017}
\end{equation}
for all $n \geq 2$.
Indeed, apart from the two groups 7T4 and 8T47 in the standard classification of Butler and McKay~\cite{BM1983}, this is a consequence of the discussion following \cite[Theorem 1.12]{CD2023}. 
However, using \texttt{SageMath}~\cite{Sag2025}, we find that 7T4 has index at least 3 in the sense of Malle's conjecture and that 8T47 is imprimitive. Thus, applying \cite[Theorem 2]{Bha2025} to the former and Widmer's result \cite{Wid2011} to the latter, we derive the bound \eqref{AltMain} above.

The upshot is that, for Conjecture \ref{vdW}, it suffices to bound the number of $A_n$ polynomials.

\subsection{Methods}
Denote by $N_{A_n, n}(B)$ the number of polynomials as in \eqref{monic} with all $|a_j| \leq B$ and Galois group $A_n$. In view of \eqref{AltMain}, for the purposes of proving Conjecture~\ref{vdW} it suffices to prove that
\[
N_{A_n, n}(B) = o(B^{n-1})
\qquad (B \to \infty).
\]
Recall that a squarefree polynomial $f$ as in \eqref{monic}, say with roots $\alpha_1, \ldots, \alpha_n$, has square \emph{discriminant}
\[
\Del = \Del(f) = \Del(a_1, \ldots, a_n) \coloneqq \prod_{1 \leq i < j \leq n} (\alpha_i - \alpha_j)^2
\]
if and only if $G_f \leq A_n$; see \cite[Theorem~7.4.1]{Cox}. Thus the polynomials counted by $N_{A_n, n}$ form a subset of the set of polynomials as in \eqref{monic} with $|a_j| \le B$ for all $j$ whose discriminant is a nonzero square. It is the cardinality of this latter set that we bound.

Our approach adapts and unifies three main tools: a square sieve, the $q$-analogue of van der Corput differencing for exponential sums, and a topological $h$-principle.  Our starting point is inspired by the work of Heath-Brown and Pierce \cite{HBP2012}, who used the power sieve and the $q$-analogue of van der Corput differencing to break past a barrier for counting integer solutions to polynomial equations of the form
\[
F(x_1, \ldots, x_n) = y^r
\]
under suitable conditions. In particular, they assumed that the vanishing of $F$ defines a smooth, projective variety. Since the discriminant variety $\Del = 0$ is far from smooth, their results cannot be applied directly to count square values of $\Del$.

The $q$-analogue of van der Corput's method originates in 1978 work of Heath-Brown \cite{HB1978} and is an iterative process using subtraction between shifted exponential sums to obtain cancellation while respecting the modulus $q$.  Motivated by Heath-Brown and Pierce \cite{HBP2012}, we carefully use an adapted version to difference over one of our moduli; with an interplay of analytic and number theoretic tools, we reduce the problem to bounding a complete exponential sum.

The complete exponential sum we obtain involves the Legendre symbol applied to the discriminants of polynomials, since the Legendre symbol is crucial for detecting square discriminants. The relationship of this to analytic number theory over function fields is revealed by Pellet's formula from 1878 \cite{Pel1878} (which has an intriguing history of its own):
\begin{equation}
\label{Pellet}
\mu_p(f) = (-1)^n \left(
\frac{\Del(f)}{p} \right) \qquad (p \ge 3 \text{ prime}),
\end{equation}
where $\mu_p(f)$ is as in \eqref{eq:mobius} and $f$ is of degree $n$ over $\bF_p$. Applying this identity, the needed bound for the complete exponential sum becomes a variant of  Chowla's conjecture. Pellet's formula is not strictly needed in the proof because our proof of this exponential sum bound is expressed in terms of the Legendre symbol of the discriminant. However, we include it because it is essential for how we think about the exponential sum bound (in particular, as a result in function field number theory).

\bigskip

The function field result relies on the following peculiar observation: Since the M\"obius function $\mu(n)$  takes integer values, we have $\mu(n) \equiv \mu(n)^2 \bmod 2$, where $\mu(n)^2$ is the indicator function of the squarefree numbers. This observation is likely of no use in analytic number theory, where many problems involve estimating the size of various sums involving the M\"obius function, since a congruence for a function can help to evaluate the congruence class of a sum involving that function but not its size. 

In the function field context, again we have $\mu_p(f) \equiv \mu_p(f)^2\bmod 2$, where $\mu_p(f)^2$ is the indicator function of the squarefree polynomials. Again this observation is seemingly of no use. However, this is not so.

The reason is that a wide range of problems in analytic number theory over function fields can be tackled by means of \'{e}tale cohomology, where the key step is showing vanishing of some \'{e}tale cohomology group. \'{E}tale cohomology is computed over the $\ell$-adic numbers for a prime $\ell$, and one can take $\ell=2$, for which it makes sense to talk about the cohomology groups mod $2$.

When we do this, the distinction between the M\"obius function and its square disappears. That is, one can take the analytic number theory techniques that would apply to the sum
\begin{equation}
\label{squared-chowla-sum} 
\bE_{g \in \cM_n(\bF_p)} \, \mu_p(g)^2 \mu_p(g+h)^2 e_p(g \cdot \tet), 
\end{equation}
consider analogues of them in the cohomological setting, and apply these to compute the relevant cohomology groups modulo $2$. To tackle \eqref{squared-chowla-sum} by methods of analytic number theory, one would sieve out the non-squarefree terms.  This reduces the problem to estimating the sum of the exponential $e_p(g \cdot \tet)$ in arithmetic progressions, which is straightforward.  However, there is, in fact, an analogue of a sieve in the cohomological setting, known as Vasilliev's method in topology.

Vasilliev's method has already been applied to problems of this type by Aumonier~\cite{Aumonier}. Using this method, he proved an h-principle, which relates the cohomology of a space of polynomials satisfying conditions to the cohomology of a space of continuous functions satisfying conditions (which often appears as the topological analogue of an Euler product in number theory, such as the Euler product that appears in the main term of the $g=0$ case of \eqref{squared-chowla-sum}). Aumonier's result implies that the cohomology groups relevant for \eqref{eq:TC} are isomorphic to the cohomology groups of the space of continuous functions with coefficients in a certain local system, and one can check that these vanish. The one caveat is that Aumonier's result does not include a topological analogue of the exponential term in \eqref{eq:TC}, but to get around this, we restrict attention to arithmetic progressions where the exponential term is constant.

\subsection{Other related problems}
Our findings in the non-monic setting are of a similar strength; these will be discussed in detail in \S \ref{nm}.

In this article we have been discussing the `large box model' of probabilistic Galois theory in which the random polynomial has fixed degree but the coefficients grow. The `large degree model' (or `restricted coefficient model') considers the complementary setting in which the set of allowed coefficients is fixed and the degree is large. Here it is a widely held belief, going back at least to Odlyzko and Poonen \cite{OP1993}, that a random polynomial is irreducible asymptotically almost surely as long as the constant coefficient vanishes with probability $0$ (and the measures by which the coefficients are chosen are not too concentrated). Despite impressive progress \cite{BSHKP, BSKK2023, BSK2020, BV2019}, the conjecture has not been settled. Interestingly, in setups where it has been settled, it has been shown that the probability that the Galois group contains $A_n$ tends to $1$ as the degree $n \to \infty$; see \cite[Theorem~6]{BSKK2023}. Ruling out the alternating group as a likely Galois group is an open problem in this setting as well. The intermediate setting in which the degree is large and the coefficients also grow (but arbitrarily slowly) has been studied in \cite{BSG2025}, and here the alternating group can be ruled out. More recently, random multiplicative coefficients have been investigated in \cite{KM, VX}.
Another natural notion of a random polynomial to consider is the characteristic polynomial of a random matrix \cite{AOD, Ebe2022, FJSS2023, Riv2008}. In each of these settings, the Galois group is subject of scrutiny as well. 

There has also been some work in the analogous setting of polynomials defined over function fields \cite{BSEM2024, Ent2025, EP2024}. Interestingly, a polynomial version of Chowla's conjecture was used in the large box model over function fields to show conditionally that the Galois group $S_n$ is generic for irreducible polynomials \cite{EP2024}.

Instead of counting polynomials, one can also count number fields with prescribed Galois group \cite{Alb2020, ALOWW, Bha2025, CD, Dum2018, FK, KP2023, Mak1985, MTTW, Wan2021, Wri1989}, where there is Malle's conjecture \cite{LS, Mal2002, Mal2004}. This counting problem feeds into the problem of counting number fields of fixed degree $n$ with absolute discriminant up to $X$, where a folklore conjecture predicts an asymptotic main term of $c_n X$, for some $c_n > 0$. See \cite{Eight2024, Bha2005, Bha2010, BSW2022, DH1971, LOT2022} for some results in this direction, as well as \cite{LO}.

The connection between these problems and others in arithmetic statistics is nicely described in Pierce's ICM proceedings \cite{Pie2022}. We also refer the reader to \cite{BSTTTZ2020} for a wonderful illustration of the significance of counting square discriminants. When $n \ge 5$, finally, by a result of Zarhin~\cite{Zar2000}, the quantity $E_n^\dag(B)$ is also an upper bound for the count of hyperelliptic Jacobians with complex multiplication. These are just some of the connections that our counting opens further.

\subsection*{Notation}

We adopt the Vinogradov and Bachmann--Landau notations, as we now describe. If $f$ and $g$ are complex-valued functions, then we write $f \ll g$ or $f = O(g)$ if $|f| \le C|g|$ pointwise, for some constant $C>0$. We write $f \asymp g$ if $g \ll f \ll g$, and $f = o(g)$ if $f/g \to 0$ in some specified limit. We use a subscript to indicate that the implied constant $C$ is allowed to depend on certain variables.

Note that in the monic case, the various operations we define (such as the polynomial dot product) omit interaction with the leading coefficient.  This ensures that the crucial step of Poisson summation on $\cM_n(\cR)$ works out correctly.  

\subsection*{Development of this version, use of AI, and human collaborations}

Developing this paper has been an exciting process with human interactions at its core.  A first version of this paper by the first six authors was posted to the math arXiv (without using AI) after discussions with Sawin and Shusterman about their promising progress on the twisted Chowla estimates.  After posting, GPT-5.6 Sol was used as an audit, which provided some initial insights into the twisted Chowla statement $\TC(n,\tau)$. After further interactions with GPT-5.6 Sol and GPT-6 Astra, we obtained a proof of Theorem~\ref{ChowlaMain}.  Sawin and Shusterman were contacted, and due to their insight and previous work, eagerly joined to verify, simplify, streamline and rewrite the proof (incorporated as the present \S \ref{sec:Chowla-estimates}).  This led to discussions among all of us, creating the current version.  Human interactions not only shaped this work, they are the essence of this entire narrative.
 
\subsection*{Organisation}

In \S \ref{ss}, we present a variant of Heath-Brown and Pierce's square sieve. In \S \ref{passage}, we establish Theorem \ref{MainThm} assuming a technical proposition. We then prove that proposition in \S \ref{sec:pPoisson}. The proof of Theorem~\ref{ChowlaMain} is given in \S \ref{sec:Chowla-estimates}. Finally, in \S \ref{nm}, we discuss the non-monic version of Van der Waerden's conjecture.

\subsection*{Acknowledgements}

We thank the Mathematisches Forschungsinstitut Oberwolfach for favourable working conditions.  Thank you to Manjul Bhargava, Ofir Gorodetsky, Lillian Pierce, Nikita Shulga, and Frank Thorne for helpful suggestions in preparing this work.

\subsection*{Funding}
The first author was partially supported by NSF CAREER DMS-2237937 and the Gregg Zeitlin Early Career Professorship.
The fourth author was supported by the Dutch Research Council (NWO) through the grant number OCENW.M20.233.
The fifth author gratefully acknowledges the support of the Dutch Research Council (NWO) through the Veni grant `New methods in arithmetic statistics' and the ERC Starting Grant `CEDE'.
The sixth author was partially supported by NSF grant DMS-2618934 and by the Office of the Vice Chancellor for Research at the University of Wisconsin-Madison with funding from the Wisconsin Alumni Research Foundation.
The eighth author's research is co-funded by the European Union (ERC, Function Fields, 101161909).

\subsection*{Rights}

For the purpose of open access, the authors have applied a Creative Commons Attribution (CC-BY) licence to any Author Accepted Manuscript version arising from this submission.

\section{A flexible square sieve}
\label{ss}

An important tool in the proof of Theorem~\ref{MainThm} is a refinement of the square sieve of Heath-Brown and Pierce \cite[Lemma~1]{HBP2012} in the case that the support of the weight, denoted $\omega$ there, is \emph{polynomially} --- instead of exponentially --- bounded; indeed, since the discriminant $\Delta(f)$ is a polynomial (of degree $2n-2$) in the coefficients of $f$, this suffices for our purposes, whereas the wider range of parameters in which the resulting refined square sieve can be applied is key in the proof of Theorem~\ref{SquareVersion} below.

The polynomially bounded support is easily incorporated into the proof of Heath-Brown and Pierce, leveraging the fact that the number of prime divisors of $m$ is $o(\log V)$ if \mbox{$m \le V^\kap$}.  Moreover, since $\kap$ is fixed, the convergence implied by the $o(\cdot)$ asymptotic notation is not claimed to be uniform in $\kap$. 
Here and unless otherwise noted, the Dirichlet characters $\chi_{j}(m)$ are the Jacobi symbols $\bigl(\frac{m}{j}\bigr)$.

\begin{lemma}
\label{lSquareSieve}
Fix $\kappa>0$. Then there exists $C = C(\kap) > 0$ such that the following holds.

Let $\mathcal{U}$ and $\mathcal{V}$ be disjoint sets of odd primes and define $$\mathcal{A} = \{ u v : u \in \mathcal{U}, \ v \in \mathcal{V}\}.$$ Set $A = \# \mathcal{A}$, $U = \# \mathcal{U}$ and $V = \# \mathcal{V}$. Assume that $V \geq 10$ and
\[
A \gg_{\kappa} \log(V) \cdot \max\{U + V, V^2\}.
\]
Let $\omega: \bZ \rightarrow \bR_{\geq 0}$ be such that $\omega(m) = 0$ for all $|m| \geq V^\kappa$. Then
\begin{align*}
\sum_{m \neq 0} \omega(m^2) &\leq C A^{-1} \sum_m \omega(m) \\
&\quad + C A^{-2} \sum_{v, v' \in \mathcal{V}} \sum_{u \neq u' \in \mathcal{U}} \left| \sum_m \omega(m) \chi_{uv}(m) \chi_{u'v'}(m) \right| \\
&\quad + C U A^{-2} \sum_{v \neq v' \in \mathcal{V}} \left| \sum_m \omega(m) \chi_v(m) \chi_{v'}(m) \right|.
\end{align*}
\end{lemma}

\begin{remark}
We can conveniently pick $\kappa > 0$ as a function of $n$. 
Then $C$ depends only on $n$, which is a constant for us. 
\end{remark}

\begin{proof}
We will give lower and upper bounds for
$$
\Sigma := \sum_m \omega(m) \left| \sum_{q \in \mathcal{A}} \chi_q(m) \right|^2.
$$
For the lower bound, we note that each $m$ is summed with non-negative weight. Moreover, if $m = t^2$ is a non-zero square and $\omega(m) \neq 0$, then
\begin{equation}
\label{eSquareLower}
\sum_{q \in \mathcal{A}} \chi_q(m) = \sum_{q \in \mathcal{A}} \chi_q(t^2) = \sum_{\substack{q \in \mathcal{A} \\ \gcd(q, m) = 1}} 1 = A - \sum_{\substack{q \in \mathcal{A} \\ \gcd(q, m) \neq 1}} 1 \gg_{\kappa} A.
\end{equation}
Indeed, this last step follows since $q = uv \in \mathcal{A}$ can have non-trivial greatest common divisor with $m$ for two distinct reasons: since $u$ and $v$ are primes, either $u$ divides $m$ or $v$ divides $m$. Therefore the second sum is at most
\begin{equation}
\label{eOmegaBound}
\nu(m) (U + V) = o_\kappa(\log(V) \cdot (U + V)),
\end{equation}
where $\nu(m)$ denotes the number of distinct prime divisors of $m$ and where we used that $\omega(m)$ is non-zero only if $m < V^\kappa$. Inserting \eqref{eOmegaBound} gives \eqref{eSquareLower}, upon using our assumption that $A \gg_{\kappa} \log(V) \cdot (U + V)$.
Thus,
\begin{equation}
\label{eSigmaLower}
\Sigma \gg_{\kappa} A^2 \sum_{m \neq 0} \omega(m^2).
\end{equation}

For the upper bound, we expand $\Sigma$ as
\begin{alignat}{2}
\Sigma &= &&\sum_{q, q' \in \mathcal{A}} \sum_m \omega(m) \chi_q(m) \chi_{q'}(m) \nonumber \\
&= &&\sum_{q \in \mathcal{A}} \sum_m \omega(m) \chi_q(m)^2 \, + \sum_{\substack{q \neq q' \in \mathcal{A} \\ \gcd(q, q') = 1}} \sum_m \omega(m) \chi_q(m) \chi_{q'}(m) \nonumber \\
& && + \sum_{\substack{q \neq q' \in \mathcal{A} \\ \gcd(q, q') \neq 1}} \sum_m \omega(m) \chi_q(m) \chi_{q'}(m). \label{eUpper}
\end{alignat}
The first term in \eqref{eUpper} is at most 
\begin{equation}
\label{eUpper1}
A \sum_m \omega(m).
\end{equation}
We break the third term in \eqref{eUpper} into the two different sums
\begin{align*}
S(\mathcal{U}) &:= \sum_{v \in \mathcal{V}} \sum_{u \neq u' \in \mathcal{U}} \sum_m \omega(m) \chi_{uv}(m) \chi_{u'v}(m), \\
S(\mathcal{V}) &:= \sum_{u \in \mathcal{U}} \sum_{v \neq v' \in \mathcal{V}} \sum_m \omega(m) \chi_{uv}(m) \chi_{uv'}(m).
\end{align*}
We may group $S(\mathcal{U})$ with the second term in \eqref{eUpper} and apply the triangle inequality to get
\begin{equation}
\label{eUpper2}
\sum_{v, v' \in \mathcal{V}} \sum_{u \neq u' \in \mathcal{U}} \left| \sum_m \omega(m) \chi_{uv}(m) \chi_{u'v'}(m) \right|.
\end{equation}
It remains to treat $S(\mathcal{V})$. We write
\begin{align*}
S(\mathcal{V}) &= \sum_{u \in \mathcal{U}} \sum_{v \neq v' \in \mathcal{V}} \sum_{u \nmid m} \omega(m) \chi_{uv}(m) \chi_{uv'}(m) \\
&= \sum_{u \in \mathcal{U}} \sum_{v \neq v' \in \mathcal{V}} \sum_{u \nmid m} \omega(m) \chi_{v}(m) \chi_{v'}(m) \\
&= U \sum_{v \neq v' \in \mathcal{V}} \sum_m \omega(m) \chi_v(m) \chi_{v'}(m) \\
&\qquad - \sum_{u \in \mathcal{U}} \sum_{v \neq v' \in \mathcal{V}} \sum_{u \mid m} \omega(m) \chi_v(m) \chi_{v'}(m) \\
&=: M(\mathcal{V}) - E(\mathcal{V}).
\end{align*}

For $E(\mathcal{V})$, note that the term $m = 0$ does not contribute due to the presence of $\chi_v(m) \chi_{v'}(m)$. Then we bound $E(\mathcal{V})$ trivially by
\begin{equation}
\label{eUpper3}
|E(\mathcal{V})| \leq V^2 \sum_{m \neq 0} \omega(m) \nu(m) \ll_\kappa V^2 \log V \sum_m \omega(m),
\end{equation}
where we used the estimate $\nu(m) \ll_\kappa \log V$, which follows from the fact that $\omega(m) = 0$ for $|m| \geq V^\kappa$ and standard bounds for $\nu(m)$.

Combining the upper bounds \eqref{eUpper1}, \eqref{eUpper2} and \eqref{eUpper3}, we get that
\begin{align*}
\Sigma &\ll_\kappa (A + V^2 \log V) \sum_m \omega(m) \\&\quad + \sum_{v, v' \in \mathcal{V}} \sum_{u \neq u' \in \mathcal{U}} \left| \sum_m \omega(m) \chi_{uv}(m) \chi_{u'v'}(m) \right| + |M(\mathcal{V})|.
\end{align*}
Since $V^2 \log V \ll_{\kappa} A$ by assumption, we may rewrite this as
$$
\Sigma \ll_\kappa A \sum_m \omega(m) + \sum_{v, v' \in \mathcal{V}} \sum_{u \neq u' \in \mathcal{U}} \left| \sum_m \omega(m) \chi_{uv}(m) \chi_{u'v'}(m) \right| + |M(\mathcal{V})|.
$$
Comparing this with \eqref{eSigmaLower} gives the lemma.
\end{proof}
\section{The passage to function fields and deduction of Theorem~\ref{intro-main}}
\label{passage}

We now set up our passage to function fields and explain how the statement $\TC(n, \tau)$ enters the problem. The proof of Theorem~\ref{MainThm} and the deduction of Theorem~\ref{intro-main} from Theorems~\ref{MainThm} and \ref{ChowlaMain} are given at the end of this section.

Let $w \colon \bR^n \rightarrow \bR_{\geq 0}$ be a smooth, non-negative weight supported on $[-2, 2]^n$, such that $w = 1$ on $[-1, 1]^n$ and
\[
\frac{\partial^\bbet}{\partial \bx^\bbet} w(\bx) \ll_\bbet 1
\]
for all multi-indices $\bbet = (\beta_1, \dots, \beta_n)$.
Set $w_B(\bx) := w(\bx/B)$ and define
$$
\ome_B(m) := \sum_{\substack{f \in \cM_n(\bZ) \\ \Del(f) = m}} w_B(f).
$$

Our main technical result involving $\TC(n, \tau)$ (see Definition~\ref{def:TC}) is as follows.

\begin{prop}
\label{pPoisson}
Let $\tau \in (0,n]$ be such that $\TC(n, \tau)$ holds, and let $q_1, q_2 \in \bN$ be odd and coprime with $q_2 \le B$. Assume that $q_1$ is a product of two primes $pp'$ satisfying $p < p' < 2p$. 
Define
$$
T(q_1, q_2) = \sum_{f \in \bZ^n} w_B(f) \chi_{q_1}(\Del(f)) \chi_{q_2}(\Del(f)),
$$
where $\chi_{q_1}$ and $\chi_{q_2}$ are Jacobi symbols. Then
\[
T(q_1, q_2) \ll_{n,\tau} q_2^{n/2} B^{n/2} + q_1^{-\tau/2} B^{n/2} \max\{ q_1^{n/2}, B^{n/2} \}.
\]
\end{prop}

We defer the proof of Proposition~\ref{pPoisson} to \S \ref{sec:pPoisson}.

In the following theorem, we let $|f|$ denote the sup-norm of the coefficient list of $f$.

\begin{thm} 
\label{SquareVersion}
Let $n \ge 3$ be an integer, and denote 
\begin{equation}
\label{eq:N(B)}
N_n(B) := \# \{ f \in \cM_n(\bZ) : |f| \le B, \ \ 0 \ne \Del(f) = \square\}.
\end{equation}
Suppose $\TC(n, \tau)$ holds for some $\tau$ in the range
\begin{equation}
\label{etauconstraints}
    \frac{4n}{n+2} < \tau < \frac{n(n+6)}{2n+3}.
\end{equation}
Then
\begin{equation}
\label{eq:NnBbound}
N_n(B) \ll_{n, \tau} B^{n-1-\frac{n (\tau - 4) + 2 \tau}{2 (n (n+2) - \tau(n+1))}} (\log{B})^2 = o(B^{n-1}),
\end{equation}
as $B \to \infty$.
\end{thm}

\begin{proof}
Let $1 < \del < \frac32$ and $\frac12 < \alp < 1$ be parameters to be chosen later in terms of $n$ and $\tau$, and introduce
\[
Q := B^\del, \qquad U_0 := Q^\alp, \qquad V_0 := Q^{1-\alp}.
\]
Define 
\begin{align*}
\mathcal{U} &:= \{p \text{ prime} : U_0/2 \leq p \leq U_0 \}, \\
\mathcal{V} &:= \{p \text{ prime} : V_0/2 \leq p \leq V_0 \}.
\end{align*}

Recall that $\Delta(f)$ is a polynomial of degree $2n-2$ in the coefficients of $f$. Hence $\Delta(f) \ll_n B^{2n-2}$ for any $f \in \cM_n(\bZ)$ with $|f| \le B$. Thus, one may verify that the conditions of Lemma~\ref{lSquareSieve} are satisfied for $B$ sufficiently large, and we get 
\begin{equation*}
N_n(B) \le \sum_{m \ne 0} \ome_B(m^2) \ll_{n, \alp} \Sigma_T + \Sigma_M + \Sigma_P,
\end{equation*}
where
\begin{align*}
\Sigma_T &:= B^n Q^{-1} (\log{Q})^2, \\ 
\Sigma_M &:= \frac{(\log{Q})^4}{Q^2} \sum_{v, v' \in \mathcal{V}} \sum_{u \neq u' \in \mathcal{U}} \left| \sum_{f \in \cM_n(\bZ)} w_B(f) \chi_{uv}(\Del(f)) \chi_{u'v'}(\Del(f)) \right|, \\
\Sigma_P &:= \frac{(\log{Q})^3}{Q^{2 - \alpha}} \sum_{v \neq v' \in \mathcal{V}} \left| \sum_{f \in \cM_n(\bZ)} w_B(f) \chi_v(\Del(f)) \chi_{v'}(\Del(f)) \right|
\end{align*}
are the trivial leading, main sieve and prime sieve terms, respectively. 

We handle $\Sigma_M$ and $\Sigma_P$ using Proposition~\ref{pPoisson}. For $\Sigma_M$, choosing
\[
q_1 = uu' \asymp Q^{2\alpha}, \qquad
q_2 = vv' \asymp Q^{2-2\alpha}
\]
in Proposition~\ref{pPoisson}, we obtain 
$$
\Sigma_M \ll \max_{q_1, q_2} |T(q_1, q_2)| \ll (Q^{(1-\alpha)n} + Q^{\alpha(n-\tau)})B^{n/2}
$$
as long as $2\delta(1-\alpha) \le 1$ (so that $B \ge q_2$). We will assume this from now on.

For $\Sigma_P$, take 
\[
q_1 = vv' \asymp Q^{2-2\alpha}, \qquad
q_2 = 1.
\]
Proposition~\ref{pPoisson} then yields
\[
\Sigma_P \ll \frac{\log{Q}}{Q^{\alpha}} \max_{q_1, q_2} |T(q_1, q_2)| \ll \frac{\log{Q}}{Q^{\alpha}} \Big( B^{n/2} + \frac{B^n}{Q^{(1-\alpha)\tau}} \Big),
\]
since $B \geq q_1$ by the assumption $2\delta(1-\alpha) \leq 1$.

Since $\tau > 1$, we have $\alpha + (1-\alpha)\tau > 1$, so that
\begin{align}
N_n(B) 
\label{eNB2}
&\ll (B^{n/2}Q^{-1} + Q^{(1-\alpha)n} + Q^{\alpha(n-\tau)}) B^{n/2} (\log{B})^2,
\end{align}
and we want to minimise the right-hand side subject to the conditions $1 < \delta < 3/2$ and $1/2 < \alpha < 1$ and $2\delta(1-\alpha) \leq 1$.
Minimising over $\alpha$ corresponds to balancing the second and third term in \eqref{eNB2}, which happens when $\alpha = n/(2n-\tau)$, so that
\begin{equation}
\label{eNB3}
    N_n(B) \ll (B^{n/2}Q^{-1} + Q^{n(n-\tau)/(2n-\tau)}) B^{n/2} (\log{B})^2.
\end{equation}
Balancing the remaining terms yields
\begin{equation}
\label{edelta}
    \delta = \frac{n(2 n - \tau)}{2 (n (n+2) - \tau (n+1))} = 1 + \frac{n(\tau-4)+2\tau}{2(n(n+2)-\tau(n+1))}.
\end{equation}
By \eqref{etauconstraints}, it is now easy to check that
\[
1/2 < \alpha < 1,
\qquad
1<\delta<3/2,
\qquad
2\delta(1-\alpha) \leq 1.
\]
Plugging \eqref{edelta} into \eqref{eNB3} concludes the proof.
\end{proof}

We are now ready to prove Theorem~\ref{MainThm}.

\begin{proof}[Proof of Theorem~\ref{MainThm}]
Define the counting functions
\begin{align*}
P_n(B) &:= \# \{f \in \cM_n(\bZ) : |f| \le B, \, G_f \not \leq A_n, \, G_f \neq S_n, \, G_f \textup{ primitive}\} \\
I_n(B) &:= \# \{f \in \cM_n(\bZ) : |f| \le B, \, G_f \textup{ imprimitive}\} 
\end{align*}
and define $N_n(B)$ as in \eqref{eq:N(B)}. 

Then
\begin{equation}
\label{eq:EnBsplit}
E_n(B) \leq P_n(B) + I_n(B) + N_n(B).
\end{equation}

By Widmer's result \cite{Wid2011} (see \cite[Equation (5)]{Bha2025}) we have $I_n(B) \ll B^{n/2+2} \ll B^{n-3/2}$, where we used the assumption $n \ge 7$.
For $n \ge 10$, we have $P_n(B) \ll B^{n-2}$ by \cite[Corollary 3(a)]{Bha2025}.

Now suppose $n \in \{7,8,9\}$. Using \texttt{SageMath}, one may verify that the complete list of primitive permutation groups $G < S_n$ of index $\ind(G) \le 2$ consists of $A_7$, $A_8$, $A_9$, and the groups labelled 7T5 (which is $\GL_3(2)$ in its action on the seven non-zero vectors of $\bF^3_2$) and 8T48 (the affine group $\AGL_3(2)$ in its standard action) in the classification of Butler and McKay~\cite{BM1983}. Each of these five groups has even parity. Therefore, polynomials with Galois group among these five groups are counted by $N_n(B)$. 

On the other hand, polynomials $f$ with Galois group $G$ of index at least $3$ contribute $O(B^{n-8(n-1)/(3(n+6))})$, by Bhargava \cite[Theorem 2]{Bha2025}; to arrive at this bound, we used (in the notation there) $\ind(G) \ge 3$ in the numerator, $a(G) = (n+2)/4$, and the inequalities $u, 1/\ind(G) > 0$ in the denominator.

Lastly, the exponent of $B$ in the upper bound on $N_n(B)$ in \eqref{eq:NnBbound} is strictly decreasing with $\tau$, and equals $n-1-\frac{n-6}{2(n+3)}$ for $\tau = n/2$. For each $n \ge 7$, these savings are smaller than those derived above for $P_n(B)$ and $I_n(B)$. Combined with \eqref{eq:EnBsplit} and Theorem~\ref{SquareVersion}, this concludes the proof.
\end{proof}

\begin{proof}[Proof of Theorem~\ref{intro-main} assuming Theorem~\ref{MainThm} and Theorem~\ref{ChowlaMain}]
Substituting $n=4k+v$ and $\tau = k/2$ into the right-hand side of \eqref{eq:EnBbound} yields the term $B^{n-1-g(k,v)} (\log{B})^2$ with 
\[
g(k,v) = \frac{k (4 k+v-30)-8 v}{56 k^2+30 k (v+1)+4 v (v+2)}.
\]
Since $k$ and $n$ are integers with $n \ge 4k+8$ and we also assume $k/2 \geq 4$ (so that $\tau$ lands in the range of Theorem \ref{MainThm}), it suffices to show that the minimum of $g(k,v)$ on the domain $k \ge 8$ and $v \in \{8,9,10,11\}$ exceeds $1/440$. One may check that, for fixed $v$, the function $g(k,v)$ is strictly increasing on $k \ge 8$; whereas for fixed $k$, it is decreasing in $v$. Thus plugging in $k=8, v=11$ yields the global minimum $g(8,11) = 4/1759 > 1/440$ on this domain. For bounded $v$, we have $g(k,v) \to 1/14$ as $k \to \infty$, so the power saving improves accordingly for large values of $n$. 
\end{proof}
\section{Proof of Proposition \ref{pPoisson}}
\label{sec:pPoisson}
In order to prove our main theorem, it remains to establish Proposition \ref{pPoisson}, which we will do in this section. Following Heath-Brown and Pierce \cite{HBP2012}, we implement the key procedure of $q$-van der Corput differencing. This has several technical steps and uses harmonic analytic tools, such as Poisson summation.  We carefully explain the details below.  For easy reference, the claim of the proposition is the bound
\[
T(q_1, q_2) \ll_{n,\tau} q_2^{n/2} B^{n/2} + q_1^{-\tau/2} B^{n/2} \max\{ q_1^{n/2}, B^{n/2} \}.
\]

\begin{proof}[Proof of Proposition \ref{pPoisson}]
Set $H := \lfloor B/q_2 \rfloor$ and $\mathcal{H} := \bZ^n \cap [1, H]^n$, so $|\mathcal{H}| = H^n$.  (We emphasize this difference between $H$ and $B$, as $H$ is often used as the height in other papers.)  Then
\begin{align*}
H^n T(q_1, q_2) &= \sum_{h \in \mathcal{H}} \sum_{f \in \bZ^n} w_B(f + q_2 h) \chi_{q_1}(\Delta(f + q_2h)) \chi_{q_2}(\Delta(f + q_2h)) \\
&= \sum_{f \in \mathcal{I}} \chi_{q_2}(\Delta(f)) \sum_{h \in \mathcal{H}} w_B(f + q_2 h) \chi_{q_1}(\Delta(f + q_2h)),
\end{align*}
where we may restrict the outer sum to $\mathcal{I} := [-4B, 4B]^n \cap \bZ^n$ by looking at the support of our weight $w_B$. Hence the Cauchy--Schwarz inequality yields
\begin{equation}
\label{eTCauchySchwarz}
H^{2n} T(q_1, q_2)^2 \ll_n B^n \sum_{f \in \mathcal{I}} \left| \sum_{h \in \mathcal{H}} w_B(f + q_2 h) \chi_{q_1}(\Delta(f + q_2h)) \right|^2.
\end{equation}
We re-extend the summation to all $f \in \bZ^n$. After expanding the square and interchanging sums, the iterated sum on the right-hand side of \eqref{eTCauchySchwarz} then becomes
$$
\sum_{h_1, h_2 \in \mathcal{H}} \sum_{f \in \bZ^n} w_B(f + q_2 h_1) w_B(f + q_2 h_2) \chi_{q_1}(\Delta(f + q_2h_1)) \chi_{q_1}(\Delta(f + q_2h_2)).
$$
Calling the inner sum $S(h_1, h_2)$, we see that $S(h_1, h_2) = S(h_1 - h_2, 0)$. Hence 
$$
\left|
\sum_{h_1, h_2 \in \mathcal{H}} S(h_1, h_2)
\right|
= 
\left|
\sum_{h_1, h_2 \in \mathcal{H}} S(h_1 - h_2, 0)
\right|
\leq H^n \sum_{h \in \mathcal{H}_0} |S(h, 0)|, 
$$
where $\mathcal{H}_0 := [-H, H]^n \cap \bZ^n$. 
Note that $S(0,0) \ll_n B^n$. Combining this with \eqref{eTCauchySchwarz} yields
\begin{align}
\label{eTBound}
T(q_1, q_2) &\ll H^{-n} B^{n/2} \left(H^n B^n + H^n \sum_{0 \ne h \in \mathcal{H}_0} |S(h, 0)|\right)^{1/2} \nonumber \\
&\ll H^{-n/2} B^n +  H^{-n/2} B^{n/2} \left(\sum_{0 \ne h \in \mathcal{H}_0} |S(h, 0)|\right)^{1/2}.
\end{align}

The first term above becomes the first term in Proposition \ref{pPoisson}.  It remains to bound the sum in \eqref{eTBound}, which leads to the second term in Proposition \ref{pPoisson}. From now on, we regard $q_2$ as fixed and will suppress it in our notation as much as possible. We write
\begin{equation*}
W(f) = W_{B,h}(f) = w_B(f) w_B(f + q_2h)
\end{equation*}
and
\begin{equation*}
t(f, h) = \chi_{q_1}(\Delta(f)) \chi_{q_1}(\Delta(f + q_2h)),
\end{equation*}
so that
\[
S(h, 0) = \sum_{f \in \bZ^n} t(f, h) W_{B,h}(f).
\]
Write $q=q_1$ for ease of notation, and let us identify $\bZ^n$ with $\cM_n(\bZ)$. 
Splitting into residue classes modulo $q$ and applying Poisson summation gives
\begin{align*}
S(h, 0) &= \sum_{g \in \cM_n(\bZ/q\bZ)} t(g, h) \sum_{G \in \cM_n(\bZ)} W_{B,h}(g + qG) \\
&= q^{-n} \sum_{g \in \cM_n(\bZ/q\bZ)} t(g, h) \sum_{\fG \in \cM_n(\bZ)} \hat W(\fG/q) e_q(g \cdot \fG) \\
&=  \sum_{\fG \in \cM_n(\bZ)} \hat W(\fG/q) \bE_{g \in \cM_n(\bZ/q\bZ)} t(g, h) e_q (g \cdot \fG).
\end{align*}

First suppose $q > B$.
By the uncertainty principle and the scaling properties of the Fourier transform, we can model
$
\hat W (\fG / q) \approx B^n 1_{|\fG| \ll q/B}.
$
Indeed, repeated integration by parts furnishes
\[
\hat W(f) \ll_n
B^n (1 +  B \| f \|_\infty)^{-n-1}
\qquad (f \in \bZ^n),
\]
which is \cite[Equation (29)]{HBP2012} at a different scale. To exploit this, let $j \in \bN$ be minimal such that $|\fG| \le 2^j q/B$. 
To proceed, we introduce the following notation. If $c$ is an element of a ring $\cR$, and $g \in \cM_n(\cR)$ is given by
\[
g(X) = X^n + b_1 X^{n-1} + \cdots + b_n,
\]
then
\[
c \star g (X) := X^n + c(b_1 X^{n-1} + \cdots + b_n).
\]
We also define an abelian group operation $\oplus$ on $\cM_n(\cR)$ by pointwise addition of the non-leading coefficients. 
By Sunzi's remainder theorem,
\begin{align*}
& \Bigl|
\mathop{\bE}_{g \in \cM_n(\bZ/q\bZ)} t(g, h) e_q(g \cdot \fG) \Bigr|
\\ &= 
\Bigl|
\mathop{\bE}_{\substack{g_1 \in \cM_n(\bF_p) \\ g_2 \in \cM_n(\bF_{p'})}} t(p' \star g_1 \oplus p \star g_2, h) e_q((p' \star g_1 \oplus p \star g_2) \cdot \fG)
\Bigr|
\\
&= \Bigl|
\mathop{\bE}_{\substack{g_1 \in \cM_n(\bF_p)}} t_p(p' \star g_1, h) e_p(g_1 \cdot \fG) \Bigr| 
\cdot
\Bigl|
\mathop{\bE}_{\substack{g_2 \in \cM_n(\bF_{p'})}} t_{p'} (p \star g_2, h) e_{p'}(g_2 \cdot \fG)\Bigr|,
\end{align*}
where $t_p(f,h) = \chi_p(\Del(f)) \chi_p(\Del(f + q_2 h))$, and similarly for $t_{p'}$.
Now note, for example, that if $h \notin p \bZ^n$ then
\begin{align*}
\Bigl|
\mathop{\bE}_{\substack{g_1 \in \cM_n(\bF_p)}} t_p(p' \star g_1, h) e_p(g_1 \cdot \fG) \Bigr| 
&= \Bigl| \mathop{\bE}_{\substack{g \in \cM_n(\bF_p)}} t_p(g, h) e_p((p')^{-1} \star g \cdot \fG) \Bigr| \\
& \leq \sup_{\substack{0 \ne h \in \cA_{\le n - 1}(\bF_p) \\ \tet \in \cM_n(\bF_p)}} \Bigl|  \bE_{g \in \cM_n(\bF_p)} \, \mu_p(g) \mu_p(g+h) e_p(g \cdot \tet) \Bigr|,
\end{align*}
where we used Pellet's formula in the last step. Applying $\TC(n,\tau)$, we compute that
\begin{align*}
&\sum_{0 \ne h \in \cH_0} |S(h,0)| \\
&\ll \sum_{j=1}^\infty
(2^j q / B)^n B^n (2^j)^{-n-1} (H^n q^{-\tau} + (H/p)^n (p')^{-\tau} + (H/p')^n p^{-\tau}
+ (H/q)^n)
\\
&\ll q^{n-\tau} H^n.
\end{align*}
Combining this with \eqref{eTBound} completes the proof in this case.

If $q \le B$ then, similarly,
\begin{align*}
\sum_{0 \ne h \in \cH_0} |S(h,0)|
&\ll \left(
B^n + 
\sum_{j=1}^\infty (2^j q / B)^n B^n (2^j)^{-n-1} \right) \\
&\qquad
\cdot (H^n q^{-\tau} + (H/p)^n (p')^{-\tau} + (H/p')^n p^{-\tau}
+ (H/q)^n) \\
&\ll B^n q^{-\tau} H^n,
\end{align*}
as was to be shown.
\end{proof}
\section{Chowla estimates}
\label{sec:Chowla-estimates}
To bound the sum \[\sum_{\mathbf a\in\Fp^n}
 \mu_p(f_{\mathbf a})\mu_p(f_{\mathbf a}+h)e_p(\xi\cdot\mathbf a)\]
 it will suffice to bound the sum
\begin{equation}\label{sum-in-progression}\sum_{\substack{ u\in \mathbb F_p[X] \\ \deg u \leq d}}\mu_p(f_0+gu)\mu_p(f_0+gu+h),
\end{equation}
as we can write the set of all polynomials of degree $n$ as a union of arithmetic progressions $\{ f_0+gu \mid u \in P_d(\Fp)\}$. Here $g$ is chosen so that the exponential $e_p(\xi\cdot\mathbf a)$ is constant on these arithmetic progressions. In classical number theory, $g$ is analogous to the denominator of a rational approximation to a real number $\alpha$ provided by Dirichlet's approximation theorem.

To bound \eqref{sum-in-progression}, by the Lefschetz fixed point formula, it suffices to show vanishing for the low-degree \'{e}tale cohomology of the moduli space of polynomials $u$ of degree $\leq d$ such that $f_0+gu$ and $f_0+gu+h$ are both squarefree, with coefficients in a certain local system with monodromy of order $2$. To obtain this vanishing for sufficiently large $p$, it suffices to obtain this vanishing over the complex numbers. Comparing \'{e}tale cohomology and singular cohomology, it suffices to prove the vanishing of the low-degree singular cohomology  of the set $U$ of polynomials $u$ of degree $\leq d$ over $\mathbb C$ such that $f_0+gu$ and $f_0+gu+h$ are both squarefree, with coefficients in a certain local system $\mathcal L$ with monodromy of order $2$.

To do this, we will use the result of Aumonier, which computes the cohomology of $U$ (and many other spaces like $U$) with constant coefficients. The computation of Aumonier is via an $h$-principle, which shows that the natural map from $U$ to a space $Y$ of continuous sections of a certain bundle induces an isomorphism on low-degree cohomology with constant coefficients.

We thus begin with a general result that deduces an injection on homology with coefficients in $\mathcal L$ from the isomorphism on low-degree homology with constant coefficients. This deduction uses the fact that the local system $\mathcal L$ has monodromy of order $2$ and thus reduces modulo $2$ to a trivial local system.  This allows us to claim the vanishing of the low-degree cohomology of $U$ with coefficients in $\mathcal L$ via the vanishing of the cohomology of $Y$ with coefficients in $\mathcal L$, which holds if $Y$ admits a fibration with $H$-space fiber on which the local system remains non-trivial.

Let $\chi_p$ be the quadratic character, extended by $\chi_p(0)=0$.
Pellet's formula gives
\begin{equation}\label{eq:Pellet}
 \mu_p(f)\mu_p(f+h)=\chi_p\bigl(\Disc(f)\Disc(f+h)\bigr)
\end{equation}
for monic $f$ of degree $n$ and $\deg h<n$, including when one of the
polynomials is not squarefree; see \cite[\S1]{CR2014}.

For a class $\epsilon\in H^1(Y,\Ftwo)$, let $\Z_\epsilon$ denote the
rank-one local system with monodromy $(-1)^{\langle\epsilon,\gamma\rangle}$.
Write $\Q_\epsilon=\Z_\epsilon\otimes\Q$. We always use singular homology and cohomology for topological spaces. As usual, relative homology and cohomology groups for a map are defined usings its mapping cylinder. 

\begin{lemma}\label{lem:modtwo}
Let $j:X\to Y$ be a map, with $X$ of finite CW type, and fix $\epsilon \in H^1(Y,\Ftwo)$. Suppose that, for an integer $k\ge0$,
\[
 H_q(Y,X;\Ftwo)=0\qquad(q\le k).
\]
Then the map
\[
 H_i(X;j^*\Z_\epsilon \otimes\Q)\longrightarrow H_i(Y;\Z_\epsilon\otimes\Q)
\]
is injective for $i<k$.
\end{lemma}

\begin{proof}
Since $\Z_\epsilon/2\Z_\epsilon$ is constant, the coefficient sequence
\[
 0\longrightarrow \Z_\epsilon\xrightarrow{\,2\,}\Z_\epsilon
 \longrightarrow\Ftwo\longrightarrow0
\]
makes $H_q(Y,X;\Z_\epsilon)$ $2$-divisible for $q\le k$. For $i<k$, the relative
homology sequence identifies
\[
 K_i:=\ker\bigl(H_i(X;j^*\Z_\epsilon)\longrightarrow H_i(Y;\Z_\epsilon)\bigr)
\]
with an image of $H_{i+1}(Y,X;\Z_\epsilon)$, so $K_i$ is $2$-divisible. It is also
a subgroup of a finitely generated abelian group, hence is finitely
generated. A finitely generated $2$-divisible abelian group is finite of
odd order. Thus $K_i\otimes\Q=0$, as required.
\end{proof}

The rational vanishing needed on the target comes from a local
H-space fiber.

\begin{lemma}\label{lem:Hspace}
Let $\rho:Y\to B$ be a Serre fibration with path-connected total space,
base, and fiber $F$. Suppose that $F$ is homotopy equivalent to an
H-space. If $\epsilon\in H^1(Y,\Ftwo)$ restricts non-trivially to $F$, then
\[
 H_i(Y;\Q_\epsilon)=0\qquad(i\ge0).
\]
\end{lemma}

\begin{proof}
Work in an H-space model of $F$, with multiplication $m$ and homotopy
unit $e$. Choose a based loop $\alpha$ detecting the restricted sign
class. The maps $x\mapsto m(\alpha(t),x)$ give a self-homotopy of
$m(e,-)$. Transport in the sign local system along this homotopy is
$-1$: this holds at the unit and hence everywhere, since $F$ is
connected. Homotopy invariance with local coefficients therefore
identifies a homology isomorphism with its negative. Over $\Q$ this
implies $H_*(F;\Q_\epsilon)=0$.

The homology Serre spectral sequence, with the indicated local system
on the total space, has first page
$H_p( B, H_q(F_b;\Q_{\epsilon|F_b}))$. All these groups vanish since $H_q(F_b;\Q_{\epsilon|F_b})$ vanishes, so the last page vanishes as well.
\end{proof}

Let $d \geq 5$ be odd, $n\geq d$, and $h$ a non-zero polynomial, for now over the complex numbers, of degree $<n$. Fix a non-zero polynomial $g$ of degree $<n-d$ and let $f_0$ be a monic polynomial of degree $n$. 
Let $U$ be the set of polynomials of degree $\leq d$ over $\mathbb C$ such that $\Disc(f_0+gu)$ and $\Disc(f_0+gu+h)$ are both non-zero.
Let $\mathcal L$ be the rational sign local system of the double cover
\begin{equation}
\label{eq:cover}
 z^2=\Disc(f_0+gu) \Disc(f_0+gu+h).
\end{equation}

To apply the results of Aumonier~\cite{Aumonier}, we must give an alternate description of $U$. We first express $U$ as a space of sections of a line bundle $\OO(d)$ on $\PP^1$. Then we introduce jet bundles and a subset $T$ of the jet bundle $J^1 \OO(d)$, and, in Lemma~\ref{bad-set-description}, express $U$ as the set of sections of $\OO(d)$ whose associated jet does not intersect $T$. The reason for this is that Aumonier's result applies to sets of sections of vector bundles whose associated jet does not intersect a fixed closed subset of the jet bundle.
For each nonnegative integer $m$, we can interpret polynomials of degree $\leq m$ as elements of the space $H^0(\OO(m))$ of sections of the line bundle $\OO(m)$ on $\PP^1$. For this reason, we can view $f_0$ and $h$ as element of $H^0( \OO(n))$ and $g$ as an element of $H^0( \OO(n-d))$, and $u$ as an element of $H^0(\OO(d))$, so that $f_0+gu$ and $f_0+gu+h$ are elements of $H^0(\OO(n))$. Because $f_0+gu$ and $f_0+gu+h$ are monic polynomials, their discriminants vanish if and only if they have a double root in $\PP^1$. 

For $E$ a line bundle on $\PP^1$, let $J^1 E$ be its jet bundle. Then for $u \in H^0(E)$, there is a corresponding section $j^1 u \in J^1E $ which is locally given by $(u,u')$. Multiplication by $g$ induces a map $J^1 g\colon J^1 \OO (d) \to J^1 \OO(n)$.  Let $T$ be the closed algebraic subset of $J^1 \OO(d)$ consisting of $\zeta \in J^1 \OO(d)$ such that 
\[
 j^1f_0+J^1g(\zeta)=0 \textrm{ or }j^1(f_0+h)+J^1g(\zeta)=0.
\]

\begin{lemma}
\label{bad-set-description} 
$U$ is the set of $u\in H^0(\mathcal O(d))$ such that $j^1 u $ does not intersect $T$. 
\end{lemma}

\begin{proof} 
For $\zeta= j^1 u $ we have
\[
j^1f_0+J^1g(\zeta)=j^1( f_0 + g u) 
\]
and
\[ j^1(f_0+h)+J^1g(\zeta)= j^1(f_0+h+gu).\]
$f_0+gu$ has a double root at a point if and only if  $j^1(f_0+gu)$ vanishes at that point, and the same is true for $f_0+h+gu$. 
\end{proof}

Aumonier makes an assumption on the dimension of the set $T$, which we now verify:

\begin{lemma}
\label{bad-set-dimension-bound} 
Unless $U$ is empty, $T$ has dimension at most $1$. 
\end{lemma}

\begin{proof}
Locally,
\[
 J^1g=\begin{pmatrix}g&0\\g'&g\end{pmatrix}.
\]
Its rank is $2$, $1$, or $0$ according as $\ord_x g$ is $0$, $1$, or
at least $2$. The fiber of $T$ over $x$ is therefore at most two points
where $g\ne0$, and at most two affine lines at a simple zero of $g$.
At a multiple zero the fiber of $T$ is either empty or the whole fiber of $J^1 \mathcal O(d)$, and in the second case $U$ is empty. The lines lie over finitely
many points, so $\dim  T\le1$. \end{proof}

Let
\[
 Y=\Gamma_{\mathrm{cts}}(\PP^1,J^1\OO(d)\setminus T),
 \qquad j:U\longrightarrow Y,\quad u\longmapsto j^1u.
\]
We are now ready to apply the result of Aumonier.

\begin{lemma}\label{Aumonier-lemma} Unless $U$ is empty, $j$ induces an isomorphism 
\[
 H_i(U;\Z)\xrightarrow{\ \sim\ }H_i(Y;\Z).
\]
for $i\leq \frac{d-3}{2}$. \end{lemma}

\begin{proof} 
The bundle $\OO(d)$ is $d$-jet ample: restriction to any effective
divisor of degree $d+1$ is surjective, since $H^1(\OO(-1))=0$.
As $T$ has codimension at least two in $J^1\OO(d)$ by Lemma~\ref{bad-set-dimension-bound}, The claim then follows from Aumonier's comparison
theorem \cite[Theorem~2.13]{Aumonier}.
\end{proof}

The local system $\LL$ is a local system on $U$, but Lemmas \ref{lem:modtwo} and \ref{lem:Hspace} apply to a local system on $Y$. To deal with this, we need the following:

\begin{lemma}\label{epsilon-existence} There exists a unique class $\epsilon$ in $H^1(Y, \Ftwo)$ such that the pullback of $\Q_\epsilon$ along $j$ is isomorphic to $\LL$. \end{lemma}

\begin{proof} Isomorphism classes of rank $1$ local systems with monodromy of order at most $2$ on $U$ are in bijection with $H^1(U, \Ftwo)$ by $\epsilon \mapsto \Z_\epsilon$, so it suffices to check that there is a unique class in $H^1(Y, \Ftwo)$ whose pullback along $j$ is the class in $H^1(U,\Ftwo)$ corresponding to $\LL$, which follows from the universal coefficient theorem and Lemma~\ref{Aumonier-lemma}. \end{proof}

Now we check the hypotheses of Lemma~\ref{lem:Hspace}.

\begin{lemma}\label{fibration-computation} Choose a closed disk $D\subset\PP^1$ avoiding the zeros of $gh$, and
put $C=\PP^1\setminus\Int D$. The restriction map
\[
 \rho:Y\longrightarrow
 B_D:=\Gamma_{\mathrm{cts}}(C,(J^1\mathcal O(d) \setminus T)|_{C} ).
\]
is a Serre fibration. If $U$ is non-empty, then the total space, base and fiber are all path-connected. The fiber is homotopy equivalent to an $H$-space. The class $\epsilon\in H^1(Y, \Ftwo)$ of Lemma~\ref{epsilon-existence} restricts non-trivially to the fiber.
\end{lemma}

\begin{proof} In a local coordinate on $D$, the jet
of $q_u=(f_0+g u)/h$ gives an invertible affine change of jet coordinates,
since $g$ and $h$ are nowhere zero. The forbidden jets are $(0,0)$
and $(-1,0)$, so the allowed-jet bundle is $D\times W$, where
\[
 W=\C^2\setminus\{(0,0),(-1,0)\}\simeq S^3\vee S^3.
\]
In particular, $W$ is $2$-connected.

$\rho$ is the pullback of the fibration
$\Map(D,W)\to\Map(\partial D,W)$ associated with the cofibration
$\partial D\hookrightarrow D$, and hence $\rho$ is a fibration. Let $F$ be a fiber of $\rho$. 

Every boundary map extends, and after
choosing a nullhomotopy of its boundary value $F$ is homotopy
equivalent to $\Omega^2W$. Since $\Omega^2 W$ is a connected $H$-space, $F$ is an $H$-space.

Since $U$ is an open subset of $\mathbb C^{d+1}$, if $U$ is non-empty, $U$ is path-connected. By the $i=0$ case of Lemma~\ref{Aumonier-lemma}, $Y$ is path-connected. Since restriction is surjective, $B_D$ is path-connected as well.

It remains to establish non-triviality of the restriction of $\epsilon$ to $F$. The subtlety here is that we have an explicit description of $\mathcal L$ only on $U$, and not on $Y$ and $F$. To deal with this, we find a loop in $\gamma \in \pi_1(U) $ on which $\mathcal L$ has non-trivial monodromy, whose image under the map $\pi_1(U) \to \pi_1(B_D)$ induced  by $U \to Y \to B_D$ is trivial. Then the homotopy exact sequence of $\rho$ shows the image of $\gamma$ in $\pi_1(Y)$ lies in the image of $\pi_1(F)$, and since the pullback of $\Z_\epsilon$ is $\mathcal L$, the image of $\gamma$ in $\pi_1(Y)$ has non-trivial monodromy under $\Z_\epsilon$.

To construct $\gamma$, it suffices to find $u_*$ such that $f_0+gu_*$ has one ordinary double zero at a point $x_0 \in \Int D$ and all other zeroes simple, while $f_0+gu_*+h$ has only simple zeroes. Indeed, since $f_0 +gu_*$ has an ordinary double root, $\Delta(f_0+gu)$ vanishes to first order at $u_*$, and so the hypersurface defined by $\Delta(f_0+gu)$ is smooth at $u_*$. Let $S_{u_*}$ be a small disc transverse to the hypersurface defined by   $\Delta(f_0+gu)$, and $\gamma$ be the boundary of $S_{u_*}$.

Since $\Delta(f_0+gu)$ has one simple zero on $S_{u_*}$ and $\Delta(f_0+gu+h)$ has no zeroes on $S_{u*}$, $\Delta(f_0+gu) \Delta(f_0+gu+h)$ has winding number $1$, so the cover $z^2 = \Delta(f_0+gu) \Delta(f_0+gu+h)$ defining $\LL$ has non-trivial monodromy on $\gamma$. 

By construction, $j^1 u_*$ intersects $T$ only over $x_0 \in \Int D$. Taking $S_{u_*}$ sufficiently small, $j^1 u$ intersects $T$ only over $\Int D$ for $u\in S_{u_*}$, so $j^1u$ gives a section of $J^1 \mathcal O(d) \setminus T$ over $\PP^1 \setminus B_D=C$. Thus the restriction of $j^1 u $ to $C$ defines a map $S_{u_*} \to B_D$ whose boundary is $\gamma$.

It remains to  find a parameter $u_*$ such that $f_0+gu_* $
has one ordinary double zero at $x_0$ and all other zeros simple,
while $f_0+ gu_* + h $ has only simple zeros. Let $f_u=f_0+gu$.

The affine space $V_0=\{u:j^1f_u(x_0)=0\}$ is non-empty of codimension
two. We will show that a generic point of $V_0$ does the trick, or in other words that the loci in $V_0$ where $f_u$ vanishes to order at least $3$ at $x_0$ or $f_u$ has a double zero away from $x_0$ or $f_u+h$ has a double zero each have codimension at least $3$ in $\mathbb C^{d+1}$.

The locus of $u$ where $f_u$ vanishes to order $3$ at $x_0$ has codimension $3$ since $g(x_0) \neq 0$.

Since $h(x_0) \neq 0$, it is not possible for $f_u+h$ to have a double zero at $x_0$.

For any other point $x$, $f_u$ or $f_u+h$ has a double zero at $x$ if and only if $(u, u')$ lies in the fiber of $T$ over $x$. For each point $\zeta$ in the fiber of $T$ at $x$, the locus of $u$ such that $(u(x), u'(x))=\zeta$ and $f_u$ has a double zero at $x_0$ has codimension $4$, since the conditions fix the value of $u$ and its derivative at two points and $d \geq 5$. Since $T$ has dimension $\leq 1$ by Lemma \ref{bad-set-dimension-bound}, the codimension of $u$ such that $(u,u')$ lies in some point of $T$ is at least $4-1=3$, as desired. \end{proof}

\begin{prop}\label{prop:vanishing}
In this setting,
\[
 H^i(U,\LL)=0\qquad\left(i<\frac{d-3}{2}\right).
\]
\end{prop}

\begin{proof}If $U$ is empty, the statement is automatic. Otherwise, Lemma~\ref{lem:Hspace}, whose assumptions are checked by Lemma~\ref{fibration-computation}, gives $H_*(Y;\Q_\epsilon)=0$, and
Lemma~\ref{lem:modtwo}, applied with $k= \frac{d-3}{2} $, gives
$H_i(U;\LL)=0$ for $i<(d-3)/2$. To check the relative homology vanishing hypothesis, use Lemma~\ref{Aumonier-lemma} and observe that the universal coefficient theorem shows that vanishing of relative integral homology implies vanishing of relative mod $2$ homology. $U$ is a finite CW complex since it is an algebraic variety.

Universal coefficients and self-duality of the sign local system give
the same vanishing in cohomology.\end{proof}

For a field $K$, write $P_j(K)=\{u\in K[X]:\deg u\le j\}$, including
the zero polynomial. This is a vector space of dimension $j+1$.
We begin with the uniform untwisted estimate supplied by the geometry.

\begin{prop}\label{prop:family-sum}
Let $k$ be a positive integer, put $d=2k+3$, and suppose that $n\ge d+1$.
For all sufficiently large primes $p$,
\begin{equation}\label{eq:family-sum}
 \left|\sum_{u\in P_d(\Fp)}
 \mu_p(f_0+gu)\mu_p(f_0+gu+h)\right|
 \ll_{n,k}p^{d+1-k/2},
\end{equation}
uniformly over monic degree-$n$ polynomials $f_0$ and non-zero
polynomials $g,h$ satisfying
\[
 \deg g\le n-d-1,\qquad\deg h<n.
\]
\end{prop}

\begin{proof}
To obtain one characteristic threshold and one constant for all
multipliers and shifts, form the parameter scheme over $\Z[1/2]$
\[
 B=\A^n_{f_0}\times(\A^{n-d}_g\setminus\{0\})
             \times(\A^n_h\setminus\{0\}),
\]
and let $V\subset B\times\A^{d+1}_u$ be the open subset where
\[
 \DD=\Disc(f_0+gu)\Disc(f_0+gu+h)
\]
is non-zero. The projection $\pi:V\to B$ is smooth of relative dimension
$d+1$.

The fiber is the open set in Proposition~\ref{prop:vanishing}, whose
cohomology with coefficients in the sign local system vanishes in degrees $i<k$.
Let $\LL_2$ be the anti-invariant summand of the pushforward of
$\Q_2$ along the double cover $z^2=\DD$.  For each $\mathbb C$-point $\bar b$ of $B$ we have 
\begin{equation}\label{stalk-vanishing}
 H_c^j(V_{\bar b},\LL_2)=0\qquad(j>2(d+1)-k)
\end{equation}
by 
Poincar\'e duality,
self-duality of the sign character, and Betti--\'etale comparison.

By constructibility of compactly-supported cohomology and the proper base change theorem \cite[Exp.~VI, Lemma~2.2.2 and \S2.2.3(B)]{SGA5},
 $R^j\pi_!\LL_2$
is a constructible sheaf with geometric stalks $H_c^j(V_{\bar b},\LL_2)$. For $j>2 (d+1)-k$, the support of this sheaf omits every $\mathbb C$-point of $B$, and hence is contained in finitely many primes.  There are only
finitely many relevant degrees, $0\le j\le 2(d+1)$. After removing these
primes, \eqref{stalk-vanishing} holds for every geometric point $\bar{b}$. The constructibility of the sheaf also implies that there is a uniform bound for the dimension of the stalks $H_c^j(V_{\bar b},\LL_2)$ (which is inexplicit, but can be made explicit using later Betti number bounds).

On each finite-field fiber, $\LL_2$ is pure of weight zero, with
Frobenius trace $\chi_p(\DD(u))$. The Grothendieck--Lefschetz fixed point formula and Weil II \cite[Theorem~3.3.1]{WeilII} give
\[
 \left|\sum_{u\in V_b(\Fp)}\chi_p(\DD(u))\right| =\left| \sum_{j\le2(d+1)-k}(-1)^j \operatorname{tr}(\operatorname{Frob}_p, H_c^j(V_{\bar b},\LL_2))\right| \] \[
 \le\sum_{j\le2(d+1)-k}\dim H_c^j(V_{\bar b},\LL_2)\,p^{j/2}
 \ll_{n,k}p^{d+1-\frac{k}{2}}.
\]
By \eqref{eq:Pellet}, this is \eqref{eq:family-sum}: terms with
$u\notin V_b$ contribute zero.
\end{proof}

\begin{proof}[Proof of Theorem~\ref{ChowlaMain}]

Put $d=2k+3$; the hypothesis is $n\ge2d+2$.
Fix $p$ large enough for Proposition~\ref{prop:family-sum}, a non-zero
shift $h$ of degree less than $n$, and a frequency $\xi\in\Fp^n$.
Let $\lambda:P_{n-1}(\Fp)\to\Fp$ be the coefficient linear form
$\lambda(f_{\mathbf a}-X^n)=\xi\cdot\mathbf a$.
The map
\[
 P_{n-d-1}(\Fp)\longrightarrow\Fp^{d+1},
 \qquad g\longmapsto
       \bigl(\lambda(g),\lambda(Xg),\ldots,\lambda(X^dg)\bigr)
\]
has a non-zero kernel, because $n-d>d+1$. Choose $g\ne0$ in that
kernel. Then $W=gP_d(\Fp)$ has dimension $d+1$ and lies in
$\ker\lambda$.

Partition $P_{n-1}(\Fp)$ into its $p^{n-d-1}$ cosets modulo $W$.
On the coset represented by $a$, the associated monic polynomials
have the form $X^n+a+gu$, and the phase is the constant
$e_p(\lambda(a))$. Applying Proposition~\ref{prop:family-sum} to
each coset and summing gives
\[
 \left|\sum_{\mathbf a\in\Fp^n}
 \mu_p(f_{\mathbf a})\mu_p(f_{\mathbf a}+h)e_p(\xi\cdot\mathbf a)\right|
 \ll_{n,k}p^{n-d-1}p^{d+1-\frac{k}{2}}=p^{n-\frac{k}{2}}.
\]
The uniformity in $h$ and $\xi$ follows from the uniformity of the
family estimate.
\end{proof}

\section{The non-monic setting}
\label{nm}

We now discuss the analogous problem in which the polynomials are not assumed to be monic. Again, the degree $n \ge 3$ is fixed. Let $\tilde E_n(B)$ count irreducible polynomials
\[
a_0 X^n + a_1 X^{n-1} + \cdots + a_n \in \bZ[X]
\]
with $a_0 \ne 0$ such that $|a_j| \le B$ for all $j$ and $G_f \not \cong S_n$. In this context, the strong form of Van der Waerden's conjecture asserts that
$$
\tilde E_n(B) = o(B^n)
\qquad (B \to \infty).
$$
Let $\tilde E_n^\dag(B)$ count irreducible polynomials
\[
a_0 X^n + a_1 X^{n-1} + \cdots + a_n \in \bZ[X]
\]
with $a_0 \ne 0$ such that $|a_j| \le B$ for all $j$ and $G_f \not \ge A_n$. The arguments leading to \cite[Corollary 3]{Bha2025} also give
\[
\tilde E_n^\dag(B) \ll_n B^{n-1}
\qquad (n \ge 10).
\]
This reduces the non-monic version of the strong form of Van der Waerden's conjecture to counting polynomials with square discriminant, when $n \ge 10$. 
In principle, existing methods should achieve this reduction for smaller values of $n$ as well. Among other things, this would require the development of a subtle and general resolvent theory akin to \cite[\S 4]{CD2023}. For brevity, we do not pursue this in the present article.

Let us move onto the main business of this section. Below, we write $g \cdot \tet$ for the dot product of the coefficient vectors. By a change of variables, the following definition is equivalent to the definition of $\TC(n,\tau)$ from the introduction.

\begin{defn} [Equivalent alternative definition]
Fix $n \in \bN$ and $\tau > 0$. Then $\TC(n,\tau)$ is the statement that, for all sufficiently large primes $p$,
\begin{align*}
&\displaystyle
\sup_{\substack{0 \ne a_0 \in \bF_p \\ 0 \ne h \in \cA_{\le n - 1}(\bF_p) \\ \tet \in \cA_{\le n-1}(\bF_p)}} \left|
\mathop{\bE}_{g \in \cA_{\le n-1}(\bF_p)} \, \mu_p(a_0 X^n + g) \mu_p(a_0X^n + g+h) e_p(g \cdot \tet)
\right| 
\ll_{n,\tau} p^{-\tau}.
\end{align*}
\end{defn}

\begin{thm}
\label{MainNM}
Let $n \ge 7$ be an integer, and suppose $\TC(n,\tau)$ holds for some $\tau$ in the range \eqref{tauRange}.
Then
\[
\tilde E_n(B) - \tilde E_n^\dag(B)
\ll_{n, \tau} B^{n - \frac{n (\tau - 4) + 2 \tau}{2 (n (n+2) - \tau(n+1))}} (\log{B})^2.
\]
In particular, if also $n \ge 10$, then
\[
\tilde E_n(B) \ll_{n, \tau}  B^{n - \frac{n (\tau - 4) + 2 \tau}{2 (n (n+2) - \tau(n+1))}} (\log{B})^2.
\]
\end{thm}

Combining Theorems \ref{ChowlaMain} and \ref{MainNM}, we now have the following analogue of
Theorem \ref{intro-main}, by essentially the same reasoning.

\begin{thm}
\label{IntroMainNM}
We have $\tilde E_n(B) \ll_n B^{n - 1/440}$ for all $n \ge 40$.
\end{thm}

\bigskip

The key proposition is the following variant of Proposition \ref{pPoisson}, which uses the same weight $w_B: \bR^n \to [0,\infty)$. 

\begin{prop}
\label{pPoissonNM}
Let $\tau \in (0,n]$ be such that $\TC(n, \tau)$ holds, and let $q_1, q_2 \in \bN$ be odd and coprime with $q_2 \le B$. Assume that $q_1$ is a product of two primes $pp'$ satisfying $p < p' < 2p$. Let $a_0 \in \bZ$ with $0 < |a_0| \le B$ and $(q_1, a_0) = 1$.
Define
$$
T(q_1, q_2; a_0) = \sum_{f \in \cA_{\le n-1}(\bZ)} w_B(f) \chi_{q_1 q_2}(\Del(a_0X^n + f)),
$$
where $\chi_{q_1 q_2}$ is the Jacobi symbol. Then
$$
T(q_1, q_2; a_0) \ll_{n, \tau} q_2^{n/2} B^{n/2} + q_1^{-\tau/2} B^{n/2} \max\{q_1^{n/2}, B^{n/2}\}.
$$
\end{prop}

\begin{proof}
[Proof of Theorem \ref{MainNM} assuming Proposition \ref{pPoissonNM}]
Define
\[
N(B) = \sum_{0 < |a_0| \le B} N(B; a_0),
\]
where
$
N(B; a_0) = \# \{ f \in \cA_{\le n-1}(\bZ): |f| \le B, \ \ 
0 \ne \Del(a_0 X^n + f) = \square \}.
$
Since $A_n$ polynomials have square discriminant, it suffices to prove that 
\[
N(B) \ll_{n,\tau}
B^{n - \frac{n (\tau - 4) + 2 \tau}{2 (n (n+2) - \tau(n+1))}} (\log{B})^2.
\]
Note that $\TC(n,\tau)$ holds, and that $\tau$ lies in the range \eqref{etauconstraints}.
We imitate the
proof of Theorem~\ref{SquareVersion}, applying the square sieve with
\begin{align*}
\cU &= \{ p \nmid a_0 \text{ prime}: U_0/2 \le p \le U_0 \},
\\
\cV &= \{ p \nmid a_0 \text{ prime}: V_0/2 \le p \le V_0 \}.
\end{align*}
Applying Proposition \ref{pPoissonNM} in lieu of Proposition \ref{pPoisson}, we conclude that
$$
N(B;a_0) \ll
B^{n - 1 - \frac{n (\tau - 4) + 2 \tau}{2 (n (n+2) - \tau(n+1))}} (\log{B})^2.
$$
Summing over $a_0$ completes the proof.
\end{proof}

\begin{proof}
[Proof of Proposition \ref{pPoissonNM}]
We mimic the proof of Proposition \ref{pPoisson}. This gives
\[
T(q_1, q_2; a_0) \ll H^{-n/2} B^n + H^{-n/2} B^{n/2} \left(
\sum_{0 \ne h \in \cH_0} |S(h,0; a_0)| \right)^{1/2}
\]
where, with $q = q_1$,
\[
S(h,0; a_0) = \sum_{\fG \in \cA_{\le n-1}(\bZ)} \hat W(\fG/q) \mathop{\bE}_{g \in \cA_{\le n - 1}(\bZ/q\bZ)} t(a_0 X^n + g,h) e_q(g \cdot \fG).
\]
We obtain
\begin{align*}
&\left| \mathop{\bE}_{g \in \cA_{\le n-1}(\bZ/q\bZ) } t(a_0X^n + g,h) e_q(g \cdot \fG) \right| \\ & \qquad \qquad = \left| \mathop{\bE}_{g \in \cA_{\le n-1} (\bF_p)} t_p(a_0 X^n + g, h) e_p((p')^{-1} g \cdot \fG) \right| \\
&\qquad \qquad \qquad \cdot \left| \mathop{\bE}_{g \in \cA_{\le n-1} (\bF_{p'})} t_{p'}(a_0 X^n + g, h) e_{p'}(p^{-1} g \cdot \fG) \right|.
\end{align*}
Importantly, we have assumed that $(q,a_0) = 1$; this ensures that we know the degree of $a_0 X^n + g$ modulo $q$, and hence fixes $n$ in Pellet's formula \eqref{Pellet}. If $h \notin p \bZ^{n+1}$ then, by Pellet's formula and $\TC(n,\tau)$, 
\begin{align*}
\sum_{g \in \cA_{\le n-1}(\bF_p)} t_p(a_0 X^n + g,h) e_p(g \cdot \fG) 
&\ll p^{n-\tau},
\end{align*}
and similarly for the other factor. The proof may then be completed in the same way as that of Proposition \ref{pPoisson}.
\end{proof}

\providecommand{\bysame}{\leavevmode\hbox to3em{\hrulefill}\thinspace}


\begin{thebibliography}{50}

\bibitem{Alb2020}
B. Alberts, \emph{The weak form of Malle's conjecture and solvable groups}, Res. Number Theory (2020) 6:10, 23 pp.

\bibitem{ALOWW}
B. Alberts, R. J. Lemke Oliver, J. Wang and M. M. Wood, \emph{Inductive methods for counting number fields}, J. Reine Angew. Math. \textbf{835} (2026), 3–61.

\bibitem{ABO2026}
T.~C. Anderson, A. Bertelli and E.~M. O'Dorney, \emph{Galois groups of reciprocal polynomials and the van der Waerden--Bhargava theorem}, Algebra Number Theory \textbf{20} (2026), 603--628.

\bibitem{Eight2024}
T. C. Anderson, A. Gafni, K. Hughes, R. J. Lemke Oliver, D. Lowry-Duda, F. Thorne,
J. Wang and R. Zhang, \emph{Improved bounds on number fields of small
degree}, Discrete Anal. \textbf{2024,} Paper No. 19, 24 pp.

\bibitem{Six2023}
T. C. Anderson, A. Gafni,
R. J. Lemke Oliver,
D. Lowry-Duda, G. Shakan and
R. Zhang, \emph{Quantitative Hilbert Irreducibility and Almost Prime Values of Polynomial Discriminants},
Int. Math. Res. Not. \textbf{2023,} 2188--2214.

\bibitem{AOD}
T. C. Anderson and E. M. O’Dorney,
\emph{Galois groups of random integer matrices}, 
Math. Proc. Cambridge Philos. Soc., to appear, arXiv:2506.06463. 

\bibitem{Aumonier}
A.~Aumonier, \emph{An $h$-principle for complements of discriminants}, Geom. Topol. \textbf{29} (2025), 1441--1488.

\bibitem{BBM}
L. Bary-Soroker, O. Ben-Porath and V. Matei, \emph{Probabilistic Galois Theory -- The Square Discriminant Case}, Bull. Lond. Math. Soc. 
\textbf{56} (2024), 2162--2177.

\bibitem{BSEM2024}
L. Bary-Soroker, E. Entin and E. McKemmie, \emph{Galois groups of random additive polynomials}, Trans. Amer. Math. Soc. \textbf{377} (2024), 2231--2259.

\bibitem{BSG2025}
L. Bary-Soroker and N. Goldgraber,
\emph{Full Galois groups of polynomials with slowly growing coefficients}, Bull. Lond. Math. Soc. 
\textbf{57} (2025), 941--955.

\bibitem{BSHKP}
L. Bary-Soroker, D. Hokken, G. Kozma and B. Poonen, \emph{Irreducibility of Littlewood polynomials of special degrees}, Int. Math. Res. Not. IMRN \textbf{2025}, no. 21, rnaf326.

\bibitem{BSKK2023}
L. Bary-Soroker, D. Koukoulopoulos and G.  Kozma, \emph{Irreducibility of random polynomials: general measures}, Invent. Math. \textbf{233} (2023), 1041--1120.

\bibitem{BSK2020}
L. Bary-Soroker and G. Kozma, \emph{Irreducible polynomials of bounded height}, Duke Math. J. \textbf{169} (2020), 579--598.

\bibitem{Bha2005}
M. Bhargava, \emph{The density of discriminants of quartic rings and fields}, Ann. of Math. (2) \textbf{162} (2005), 1031--1063.

\bibitem{Bha2010}
M. Bhargava, \emph{The density of discriminants
of quintic rings and fields}, Ann. of Math. (2) \textbf{172} (2010), 1559--1591.

\bibitem{Bha2025}
M. Bhargava, \emph{Galois groups of random integer polynomials and van der Waerden’s Conjecture}, Ann. of Math. \textbf{201} (2025), 339--377.

\bibitem{BSTTTZ2020}
M. Bhargava, A. Shankar, T. Taniguchi, F. Thorne, J. Tsimerman and Y. Zhao, \emph{Bounds on $2$-torsion in class groups of number fields and integral points on elliptic curves}, J. Amer. Math. Soc. \textbf{33} (2020), 1087--1099.

\bibitem{BSW2022}
M. Bhargava, A. Shankar and X. Wang, \emph{An improvement on Schmidt’s bound on the number of
number fields of bounded discriminant and small degree}, Forum Math. Sigma \textbf{10} (2022), Paper No. e86, 13 pp.

\bibitem{BV2019}
E. Breuillard and P. Varj\'u, \emph{Irreducibility of random polynomials of large degree}, Acta Math. \textbf{223} (2019), 195--249.

\bibitem{BM1983}
G. Butler and J. McKay, \emph{The transitive groups of degree up to eleven}, Comm. Algebra \textbf{11} (1983), 863--911.

\bibitem{CR2014}
D. Carmon and Z. Rudnick, \emph{The autocorrelation of the M\"obius function and Chowla's conjecture for the rational function field}, Quart. J. Math. \textbf{65} (2014), 53--61.

\bibitem{CD2020} 
S. Chow and R. Dietmann, \emph{Enumerative Galois theory for cubics and quartics}, Adv. Math. \textbf{372} (2020).

\bibitem{CD2023} 
S. Chow and R. Dietmann, \emph{Towards van der Waerden's conjecture}, Trans. Amer. Math. Soc. \textbf{376} (2023), 2739--2785.

\bibitem{CD}
S. Chow and R. Dietmann, \emph{Enumerative Galois theory for number fields}, Mathematika (special Wooley issue), to appear, arXiv:2304.11991.

\bibitem{Cox}
D. A. Cox, \emph{Galois theory}, second edition, Wiley and Sons, Hoboken, NJ, 2012.

\bibitem{DH1971} H. Davenport and H. Heilbronn, \emph{On the density of discriminants of cubic fields, II}, Proc. Roy. Soc. London Ser. A \textbf{322} (1971), 405--420.

\bibitem{WeilII}
P.~Deligne, \emph{La conjecture de Weil. II},
Publ. Math. Inst. Hautes \'Etudes Sci. \textbf{52} (1980), 137--252.


\bibitem{Die2013}
R. Dietmann, \emph{Probabilistic Galois theory}, Bull. Lond. Math. Soc. \textbf{45} (2013), 453--462. 

\bibitem{Dum2018}
E. Dummit, \emph{Counting $G$-extensions by discriminant}, Math. Res. Lett.
\textbf{25} (2018), 1151--1172.

\bibitem{Ebe2022}
S. Eberhard, \emph{The characteristic polynomial of a random matrix}, Combinatorica \textbf{42} (2022), 491--527.

\bibitem{Ent2025}
A. Entin, \emph{Galois groups of random polynomials over the rational function field}, J. London Math. Soc. (2) \textbf{111} (2025), e70061. 

\bibitem{EP2024}
A. Entin and A. Popov, \emph{Probabilistic Galois theory in function fields},
Finite Fields Appl. \textbf{98} (2024), 102466.

\bibitem{FJSS2023}
A. Ferber, V. Jain, A. Sah and M. Sawhney, \emph{Random symmetric matrices: rank distribution and irreducibility of the characteristic polynomial}, Math. Proc. Cambridge Philos. Soc. \textbf{174} (2023), 233--246.

\bibitem{FK}
E. Fouvry and P. Koymans, \emph{Malle’s conjecture for nonic Heisenberg extensions}, Ann. Inst. Fourier (Grenoble), to appear, arXiv:\allowbreak2102.09465.

\bibitem{Gal1973}
P. X. Gallagher, \emph{The large sieve and probabilistic Galois theory}, Analytic number theory (Proc. Sympos. Pure Math., Vol. XXIV, St. Louis Univ., St. Louis, Mo., 1972), pp. 91--101, Amer. Math. Soc., Providence, RI, 1973.

\bibitem{GKL2018}
A. Gomilko, D. Kwietniak and M. Lema\'nczyk, \emph{
Sarnak's conjecture implies the Chowla conjecture along a subsequence}, Ergodic theory and dynamical systems in their interactions with arithmetics and combinatorics, 237–247.
Lecture Notes in Math. \textbf{2213,}
Springer, Cham, 2018.

\bibitem{GS2020}
O. Gorodetsky and W. Sawin, \emph{Correlation of arithmetic functions over $\bF_q[T]$},
Math. Ann. \textbf{376}
(2020), 1059--1106.

\bibitem{HB1978}
D.~R. Heath-Brown, \emph{Hybrid bounds for Dirichlet $L$-functions}, 
Invent. Math. \textbf{47} (1978), 149--170.

\bibitem{HBP2012}
D. R. Heath-Brown and L. B. Pierce, \emph{Counting rational points on smooth cyclic covers}, J. Number Theory \textbf{132} (2012), 1741--1757.

\bibitem{SGA5}
J.-P.~Jouanolou, \emph{Cohomologie $\ell$-adique},
in \emph{Cohomologie $\ell$-adique et fonctions $L$ (SGA~5)},
L.~Illusie (ed.), Lecture Notes in Mathematics, vol.~589,
Springer-Verlag, Berlin--Heidelberg, 1977,
\href{https://doi.org/10.1007/BFb0096807}{251--281}.


\bibitem{KR2015}
J. P. Keating and E. Roditty-Gershon, \emph{Arithmetic correlations over large finite fields}, Int. Math. Res. Not. \textbf{2016,} 860--874.

\bibitem{KR2018}
J. P. Keating and E. Roditty-Gershon, Corrigendum to ``Arithmetic correlations over large finite fields'', Int. Math. Res. Not. \textbf{2020,} 5152--5153.

\bibitem{KM}
O. Klurman and V. Matei, \emph{Irreducibility of polynomials with random multiplicative coefficients revisited}, 	arXiv:2511.10359.

\bibitem{Kno1955}
H.-W. Knobloch, \emph{Zum Hilbertschen Irreduzibilitätssatz},
Abh. Math. Sem. Univ. Hamburg \textbf{19} (1955), 176–190.

\bibitem{KP2023}
P. Koymans and C. Pagano, \emph{On Malle’s conjecture for nilpotent groups}, Trans. Amer. Math. Soc. Ser. B \textbf{10} (2023), 310--354.

\bibitem{LO}
R. J. Lemke Oliver, \emph{Uniform exponent bounds on the number of primitive extensions of number fields}, arXiv:2311.06947.

\bibitem{LOT2022}
R. J. Lemke Oliver and F. Thorne, \emph{Upper bounds on number fields of given degree and bounded discriminant}, Duke Math. J. \textbf{171} (2022), 3077--3087.

\bibitem{LS}
D. Loughran and T. Santens, \emph{Malle's conjecture and Brauer groups of stacks},
arXiv:\allowbreak2412.04196.

\bibitem{Mak1985}
S. M\"aki, \emph{On the density of abelian number fields}, Ann. Acad. Sci. Fenn. Ser. A I Math. Dissertationes \textbf{54} (1985), 104 pp.

\bibitem{Mal2002} G. Malle, \emph{On the distribution of Galois groups}, J. Number Theory \textbf{92} (2002), 315--329.

\bibitem{Mal2004} G. Malle, \emph{On the distribution of Galois groups, II}, Exp. Math. \textbf{13} (2004), 129--135.

\bibitem{MTTW}
R. Masri, F. Thorne, W-L Tsai, and J. Wang, \emph{Malle’s conjecture for $G \times A$ with
$G = S_3, S_4, S_5$}, arXiv: 2004.04651.

\bibitem{OP1993}
A. M. Odlyzko and B. Poonen, \emph{Zeros of polynomials with 0,1 coefficients}, Enseign. Math. (2) \textbf{39} (1993), 317--348.

\bibitem{Pel1878}
A. Pellet, \emph{Sur la d\'ecomposition d’une fonction enti\`ere en facteurs irr\'eductibles suivant un module premier
$p$}, C. R. Acad. Sci. Paris \textbf{86} (1878), 1071--1072.

\bibitem{Pie2022}
L. B. Pierce, \emph{Counting problems: class groups, primes, and number fields}, Proc. Int. Cong. Math. \textbf{2022,} pp. 1940--1965,
DOI 10.4171/ICM2022/104.

\bibitem{Ram2018}
O. Ramar\'e, \emph{Chowla's conjecture: from the Liouville function to the Moebius function}, Ergodic theory and dynamical systems in their interactions with arithmetics and combinatorics, 317--323.
Lecture Notes in Math. \textbf{2213,}
Springer, Cham, 2018.

\bibitem{Riv2008}
I. Rivin, \emph{Walks on groups, counting reducible matrices, polynomials, and surface and free group automorphisms}, Duke Math. J. \textbf{142} (2008), 353--379.

\bibitem{Sag2025}
SageMath, the Sage Mathematics Software System (Version 10.7),
The Sage Developers, 2025, \url{https://www.sagemath.org}.

\bibitem{VX}
P. P. Varj\'u and M. Xu, \emph{A random polynomial with multiplicative coefficients is almost surely irreducible}, Int. Math. Res. Not. IMRN \textbf{2026}, no. 16, rnag178.

\bibitem{vdW1936}
B. L. van der Waerden, \emph{Die Seltenheit der reduziblen Gleichungen und der Gleichungen mit Affekt}, Monatsh. Math. \textbf{43} (1936), 133--147.

\bibitem{Wan2021}
J. Wang, \emph{Malle's conjecture for $S_n \times A$ for $n=3,4,5$}, 
Compos. Math. \textbf{157} (2021), 83--121.

\bibitem{Wid2011}
M. Widmer, \emph{On number fields with nontrivial subfields}, Int. J. Number Theory \textbf{7} (2011), 695--720.

\bibitem{Wri1989}
D. J. Wright, \emph{Distribution of discriminants of abelian extensions}, Proc. Lond. Math. Soc. \textbf{3} (1989), 17--50.

\bibitem{Zar2000}
Y. G. Zarhin, \emph{Hyperelliptic Jacobians without complex multiplication}, Math. Res. Lett. \textbf{7} (2000), 123--132.

\bibitem{Zyw2010}
D. Zywina, \emph{Hilbert's irreducibility theorem and the larger sieve},  arXiv:1011.6465.

\end{thebibliography}
\end{document}